\documentclass[11pt,a4paper,onesided]{article}

\usepackage{ amssymb, amsmath, amsthm, graphics, xspace, enumerate}
\usepackage{graphicx}
\usepackage[a4paper,colorlinks=true,citecolor=blue,urlcolor=blue,linkcolor=blue,bookmarksopen=true,unicode=true,pdffitwindow=true]{hyperref}
 \hypersetup{pdfauthor={Nikola Sandri\'{c}}}
\hypersetup{pdftitle={Recurrence and transience property for two
cases of stable-like Markov chains}}

\newtheorem{tm}{tm}[section]
\newtheorem{theorem}[tm]{Theorem}
\newtheorem{lemma}[tm]{Lemma}
\newtheorem{corollary}[tm]{Corollary}
\newtheorem{proposition}[tm]{Proposition}

\newtheorem{definition}[tm]{Definition}

\newtheorem{remark}[tm]{Remark}

\newcommand{\chain}[1]{\{#1_n\}_{n\geq0}}
\newcommand{\process}[1]{\{#1_t\}_{t\geq0}}
\newcommand {\R} {\ensuremath{\mathbb{R}}}
\newcommand {\ZZ} {\ensuremath{\mathbb{Z}}}
\newcommand {\N} {\ensuremath{\mathbb{N}}}
\newcommand {\CC} {\ensuremath{\mathbb{C}}}
\numberwithin{equation}{section}
\def\be{\begin{equation}}
\def\ee{\end{equation}}
\begin{document}

 \title{Recurrence and transience criteria for two
cases of  stable-like Markov chains}
 \author{Nikola Sandri\'{c}\\ Department of Mathematics\\
         Faculty of Civil Engineering, University of Zagreb, Zagreb,
         Croatia \\
         Email: nsandric@grad.hr }

 \maketitle
\begin{center}
{
\medskip

} \end{center}

\begin{abstract}
 We give  recurrence and transience criteria for two
cases of time-homogeneous  Markov chains on  the real line with
transition kernel $p(x,dy)=f_x(y-x)dy$, where $f_x(y)$ are
probability densities of symmetric distributions and, for large
$|y|$, have a power-law decay with exponent $\alpha(x)+1$, with
$\alpha(x)\in(0,2)$.

 If $f_x(y)$ is the  density of a symmetric $\alpha$-stable distribution for negative $x$ and
 the density  of a symmetric $\beta$-stable distribution
for non-negative $x$, where $\alpha,\beta\in(0,2)$, then the chain
is recurrent if and only if $\alpha+\beta\geq2.$

If the function $x\longmapsto f_x$ is  periodic and if the set
$\{x:\alpha(x)=\alpha_0:=\inf_{x\in\R}\alpha(x)\}$ has positive
Lebesgue measure,  then, under a uniformity condition on the
densities $f_x(y)$ and some mild technical conditions, the chain is
recurrent if and only if $\alpha_0\geq1.$

\end{abstract}


%
%
%
%


\section{Introduction}

Let $(\Omega,\mathcal{F},\mathbb{P})$ be a probability space and let
$\{Z_n\}_{n\in\N}$ be a sequence of i.i.d.  random variables on
$(\Omega,\mathcal{F},\mathbb{P})$ taking values in $\R^{d}$. Let us
define $S_n:=\sum_{i=1}^{n}Z_i$ and $S_0:=0$. The sequence
$\{S_n\}_{n\geq0}$ is called a \emph{random walk} with jumps
$\{Z_n\}_{n\in\N}$. The random walk $\{S_n\}_{n\geq0}$ is said to be
\emph{recurrent} if
$$\mathbb{P}\left(\liminf_{n\longrightarrow\infty}|S_n|=0\right)=1,$$ and  \emph{transient} if $$\mathbb{P}\left(\lim_{n\longrightarrow\infty}|S_n|=\infty\right)=1.$$
It is well known that every random walk is either recurrent or
transient (see   \cite[Theorem 4.2.1]{durrett}).  In the case $d=1$,
a symmetric $\alpha$-stable random walk, i.e., a random walk with
jump distribution with characteristic function
$\varphi(\xi)=\exp(-\gamma|\xi|^{\alpha})$, where $\alpha\in(0,2]$
and $\gamma\in(0,\infty)$, is recurrent if and only if $\alpha\geq1$
(see the discussion after \cite[Lemma 4.2.12]{durrett}). For
recurrence and transience properties of random walks  see
\cite{chung, durrett}.  In this paper we generalize one-dimensional
symmetric $\alpha$-stable random walks in the way that the index of
stability of  the jump distribution depends on the current position,
and we study the  recurrence and transience property of the
generalization.

From now on, using the notation from \cite{taqqu}, we will write
S$\alpha$S for the one-dimensional symmetric $\alpha$-stable
distribution. Let us denote by $\mathcal{B}(\R)$ the Borel
$\sigma$-algebra on $\R$ and by $\lambda(\cdot)$ the Lebesgue
measure on $\mathcal{B}(\R)$. Furthermore, let us introduce the
notation $f(y)\sim g(y),$ when $y\longrightarrow y_0$, for
$\lim_{y\longrightarrow y_0}f(y)/g(y)=1,$ where
$y_0\in[-\infty,\infty]$.  Recall that if $f(y)$ is the density  of
S$\alpha$S distribution with characteristic function
$\varphi(\xi)=\exp(-\gamma |\xi|^{\alpha})$, where $\alpha\in(0,2)$
and $\gamma\in(0,\infty)$, then
$$f(y)\sim c_{\alpha}|y|^{-\alpha-1},$$ when
$|y|\longrightarrow\infty,$ where $c_1=\frac{\gamma}{2}$ and
$c_{\alpha}=\frac{\gamma}{\pi}\Gamma(\alpha+1)\sin\left(\frac{\pi\alpha}{2}\right),$
for $\alpha\neq1,$ (see    \cite[Property 1.2.15]{taqqu}).

 Let $\alpha:\R\longrightarrow(0,2)$ and
$c:\R\longrightarrow(0,\infty)$ be arbitrary functions and let
$\{f_x\}_{x\in\R}$ be a family of probability densities on $\R$
satisfying:
\begin{enumerate}
  \item [(i)] $x\longmapsto f_x(y)$ is a Borel measurable function for all
  $y\in\R$ and
  \item [(ii)]$f_x(y)\sim c(x)|y|^{-\alpha(x)-1},$ for
$|y|\longrightarrow\infty$.
\end{enumerate}
We define a Markov chain $\chain{X}$ on $\R$   by the following
transition kernel
\begin{align}\label{eq:1.1}p(x,dy):=f_{x}(y-x)dy.\end{align}
Transition densities of the chain $\chain{X}$ are asymptotically
equivalent to the densities of symmetric stable distributions. We
call the Markov chain $\chain{X}$ a \emph{stable-like Markov chain}.

For Borel measurable functions  $\alpha:\R\longrightarrow(0,2)$ and
$\gamma:\R\longrightarrow(0,\infty)$,  let
$f_{(\alpha(x),\gamma(x))}(y)$ be the density  of
 a S$\alpha(x)$S distribution given by the following characteristic function $\varphi(x;\xi)=\exp(-\gamma(x)|\xi|^{\alpha(x)})$. A special case of the stable-like
 chain $\chain{X}$ is a Markov chain $\chain{X^{\alpha(x)}}$
given  by the following transition kernel
\begin{align}\label{eq:1.2}p(x,dy):=f_{(\alpha(x),\gamma(x))}(y-x)dy.\end{align}   The stable-like
chain $\chain{X^{\alpha(x)}}$ has state dependent stable jumps,
i.e., it jumps from the state $x$ by a S$\alpha(x)$S law.

The recurrence and transience problem for the stable-like chain
$\chain{X}$ (the chain given by (\ref{eq:1.1})) was already treated
in \cite{ja}. Using the Foster-Lyapunov drift criterion for
recurrence and transience of Markov chains, under a uniformity
condition on the densities $f_{x}(y)$ and some mild technical
conditions (see conditions (C1)-(C5) in \cite{ja}) it is proved that
if $\liminf_{|x|\longrightarrow\infty}\alpha(x)>1$, then the
stable-like chain $\chain{X}$
 is recurrent, and if
$\limsup_{|x|\longrightarrow\infty}\alpha(x)<1$, then the
stable-like chain $\chain{X}$
 is transient. Results in \cite{ja} give us only sufficient
conditions for recurrence and transience. In this paper we treat two
special cases of the  stable-like chain $\chain{X}$ not covered in
\cite{ja}, and  give their recurrence and transience criteria. For
recurrence and transience properties of Markov chains on general
state space see \cite{meyn}.

As already mentioned, we treat only two special cases of stable-like
chains:
\begin{enumerate}
  \item [(i)] Let $\alpha,\beta\in(0,2)$   and  $\gamma, \delta\in(0,\infty)$ be  arbitrary. Let $\chain{X^{(\alpha,\beta)}}$ be
a stable-like  chain  given by transition densities with following
characteristic functions
\begin{align}\label{eq:1.3}\varphi(x;\xi)=\left\{\begin{array}{cc}
                                                      \exp(-\gamma|\xi|^{\alpha}), & x<0 \\
                                                      \exp(-\delta|\xi|^{\beta}),&
                                                      x\geq0.
                                                    \end{array}\right.\end{align}
  \item [(ii)] Let $\alpha:\R\longrightarrow(0,2)$  and $c:\R\longrightarrow(0,\infty)$
  be arbitrary
    Borel measurable functions   and let $\{f_x\}_{x\in\R}$ be an arbitrary family of  probability
    densities
     on
$\R$ with $f_x(-y)=f_x(y)$ for all $x,y\in\R$. Furthermore, let us
assume that the function $x\longmapsto f_x$ is a periodic function
with period $\tau>0$ and that the following conditions are
satisfied:
\begin{description}
\item [\textbf{(PC1)}]the function $(x,y)\longmapsto f_x(y)$ is  continuous and strictly positive;
\item [\textbf{(PC2)}] $f_x(y)\sim c(x)|y|^{-\alpha(x)-1}$, when
$|y|\longrightarrow\infty$, for all $x\in\R$;
\item [\textbf{(PC3)}]
 $\displaystyle\lim_{|y|\longrightarrow\infty}\sup_{x\in[0,\tau]}\left|f_x(y)\frac{|y|^{\alpha(x)+1}}{c(x)}-1\right|=0;$
                          \item [\textbf{(PC4)}] $\displaystyle\inf_{x\in
                          [0,\tau]}c(x)>0$.
\end{description}
Let $\chain{X^{p}}$ be a stable-like chain, called a \emph{periodic
stable-like chain}, given by the transition kernel \be
\label{eq:1.4} p(x,dy):=f_x(y-x)dy.\ee
\end{enumerate}
Note that  $\tau$-periodicity of the function $x\longrightarrow f_x$
implies $\tau$-periodicity of the functions $\alpha(x)$ and $c(x)$.
Indeed, let $x\in\R$ be arbitrary, then, by (PC2), we have:
\begin{align*}1&=\lim_{|y|\longrightarrow\infty}f_{x+\tau}(y)\frac{|y|^{\alpha(x+\tau)+1}}{c(x+\tau)}\\&=\lim_{|y|\longrightarrow\infty}\left(f_{x}(y)\frac{|y|^{\alpha(x)+1}}{c(x)}\frac{c(x)}{c(x+\tau)}|y|^{\alpha(x+\tau)-\alpha(x)}\right)\\&=\frac{c(x)}{c(x+\tau)}\lim_{|y|\longrightarrow\infty}|y|^{\alpha(x+\tau)-\alpha(x)}.\end{align*}
Therefore,  both stable-like chains $\chain{X^{(\alpha,\beta)}}$ and
$\chain{X^{p}}$ satisfy  conditions (C1)-(C5) from \cite{ja}. In
particular, both stable-like chains $\chain{X^{(\alpha,\beta)}}$ and
$\chain{X^{p}}$ are irreducible with respect to the Lebesgue measure
(see \cite[Proposition 2.1]{ja}). Thus, we have
recurrence-transience dichotomy in both cases. Further, together
with the $\tau$-periodicity of the function $c(x)$,
 condition (PC3) implies
\be\label{eq:1.5}\sup_{x\in[0,\tau]}c(x)=\sup_{x\in\R}c(x)<\infty\ee
(see \cite[Remark 1.1]{ja}).

 From now on, we assume that the
stable-like chain $\chain{X}$ (the chain given by (\ref{eq:1.1}))
satisfies conditions (C1)-(C5). Note that, in general, this is not
the case for the stable-like chain $\chain{X^{\alpha(x)}}$  given by
(\ref{eq:1.2}) (for sufficient conditions see \cite[Proposition
5.5]{ja}). We refer the reader to \cite{ja} for more details about
conditions (C1)-(C5).

An example of the periodic stable like-chain $\chain{X^{p}}$
satisfying conditions (PC1)-(PC4) is given as follows: Let
$\alpha:\R\longrightarrow(0,2)$ be an arbitrary continuous periodic
function with period $\tau>0$ and define the family of density
functions $\{f_x\}_{x\in\R}$ on $\R$ by
$$f_x(y):=\left\{\begin{array}{cc}
                                                 \frac{1}{2}\frac{\alpha(x)}{\alpha(x)+1}     , & |y|\leq 1 \\
                                                     \frac{1}{2}\frac{\alpha(x)}{\alpha(x)+1}|y|^{-\alpha(x)-1} , &
                                                      |y|\geq 1
                                                    \end{array}\right.
                                                    $$ for all
                                                    $x\in\R.$ In
                                                    this case $c(x)=\frac{1}{2}\frac{\alpha(x)}{\alpha(x)+1}
                                                    $.

Now, let us state the main results of this paper:
\begin{theorem}\label{tm1.1}
The stable-like chain $\chain{X^{(\alpha,\beta)}}$ is recurrent if
and only if $\alpha+\beta\geq2.$
\end{theorem}

\begin{theorem}\label{tm1.2}
If the set
                                                    $\{x:\alpha(x)=\alpha_0:=\inf_{x\in\R}\alpha(x)\}$ has positive Lebesgue measure,
                                                    then the
                                                    periodic
                                                    stable-like
                                                    chain
                                                    $\chain{X^{p}}$ is
recurrent if and only if $\alpha_0\geq1$.
\end{theorem}
As a simple consequence of Theorems \ref{tm1.1} and \ref{tm1.2} we
get the following well-known recurrence and transience criterion for
the random walk case:
\begin{corollary}\label{c1.3}A S$\alpha$S  random walk on the real
line is recurrent if and only if $\alpha\geq1$.
\end{corollary}

 The same problem
was already treated, but in continuous-time case, in \cite{bjoern}
and \cite{franke,frankeerata}. In \cite{bjoern} it is proved that
the stable-like process $\process{X^{(\alpha,\beta)}}$ with the
symbol $p(x,\xi)=\gamma(x)|\xi|^{\alpha(x)}$ is recurrent if and
only if $\alpha+\beta\geq2$, where $\alpha:\R\longrightarrow(0,2)$
and $\gamma:\R\longrightarrow(0,\infty)$ are continuously
differentiable functions with bounded derivative such that
$$\alpha(x)=\left\{\begin{array}{cc}
                                                      \alpha, & x<-k \\
                                                      \beta,&
                                                      x>k
                                                    \end{array}\right.\quad\textrm{and}\quad\gamma(x)=\left\{\begin{array}{cc}
                                                      \gamma, & x<-k \\
                                                      \delta,&
                                                      x>k
                                                    \end{array}\right.$$
                                                    for  $\alpha,\beta\in(0,2)$, $\gamma,\delta\in(0,\infty)$ and  $k>0$. In
\cite{franke} the author considers the recurrence and transience
problem of the stable-like process $\process{X^{p}}$  with the
symbol $p(x,\xi)=\gamma(x)|\xi|^{\alpha(x)}$, where
$\alpha:\R\longrightarrow(0,2)$ and
$\gamma:\R\longrightarrow(0,\infty)$ are continuously differentiable
and periodic functions with bounded derivative, and proves that if
the set $\{x\in\R:\alpha(x)=\alpha_0:=\inf_{x\in\R}\alpha(x)\}$ has
positive Lebesgue measure, then the  process  is recurrent if and
only if $\alpha_0\geq1.$ Both results and technics, in \cite{bjoern}
and \cite{franke}, will be crucial in proving our results.

Now we explain our strategy of  proving the  main results. In
\cite{bjoern} it is proved that the stable-like process
$\process{X^{(\alpha,\beta)}}$ is recurrent if and only if
$\alpha+\beta\geq2$, and in \cite{bjoernschilling} it is proved that
 $\process{X^{(\alpha,\beta)}}$ can be approximated by a
sequence of Markov chains $\chain{X^{(m)}}$, $m\in\N$, such that
$\chain{X^{(1)}}\stackrel{\hbox{\scriptsize
d}}{=}\chain{X^{(\alpha,\beta)}}.$ In Theorem~\ref{tm1.1} we prove
that all chains $\chain{X^{(m)}}$, $m\in\N$, are either recurrent
 or  transient at the same time and we prove that
their recurrence property is equivalent with the recurrence property
of the stable-like process $\process{X^{(\alpha,\beta)}}$. This
accomplishes the proof of Theorem~\ref{tm1.1}. In
Theorem~\ref{tm1.2} we subordinate the periodic stable-like chain
$\chain{X^{p}}$ with the Poisson process $\process{N}$ with
parameter $1$ and, following the ideas form \cite{franke},  prove
that the sequence of strong Markov processes
$\{n^{-\frac{1}{\alpha_0}}X^{p}_{N_{nt}}\}_{t\geq0}$, $n\in\N$,
converges in distribution, with respect to the Skorohod topology, to
symmetric $\alpha_0$-stable L\'{e}vy process. Furthermore, we prove
that all the processes
$\{n^{-\frac{1}{\alpha_0}}X^{p}_{N_{nt}}\}_{t\geq0}$, $n\in\N$, are
either recurrent
 or  transient at the same time,
their recurrence property is equivalent with the recurrence property
of a   symmetric $\alpha_0$-stable L\'{e}vy process and
 recurrence properties of the process
$\{X^{p}_{N_{t}}\}_{t\geq0}$ and the periodic stable-like chain
$\chain{X^{p}}$ are equivalent. This accomplishes the proof of
Theorem \ref{tm1.2}.

Let us remark that   the idea of studying recurrence and transience
property of a Markov process
  in terms of
the property of the associated Markov chain  is studied in
\cite{tuo}.

The paper is organized as follows.  In Section 2 we introduce some
preliminary and auxiliary results which will be needed to make the
connection with results proved in \cite{bjoern} and \cite{franke}.
In Sections 3 and 4 we give proofs of Theorems \ref{tm1.1} and
\ref{tm1.2} and in Section 5 we treat discrete version of the
stable-like chains $\chain{X^{(\alpha,\beta)}}$ and $\chain{X^{p}}$
and we derive the same recurrence and transience criteria as in
Theorems \ref{tm1.1} and \ref{tm1.2}.

Throughout the paper we use the following notation.  We write
$\ZZ_+$ and $\R_+$, respectively, for nonnegative integers and
nonnegative real numbers. For $x,y\in\R$ let $x\wedge
y:=\min\{x,y\}$ and $x\vee y:=\max\{x,y\}$. For two functions $f(x)$
and $g(x)$ we write $f(x)=o(g(x)),$ when $x\longrightarrow x_0$, if
$\lim_{x\longrightarrow x_0}f(x)/g(x)=0$, where
$x_0\in[-\infty,\infty].$ Write $B_b(\R)$, $C(\R)$, $C_b(\R)$, and
$C_0(\R)$, respectively, for the sets of bounded Borel measurable
functions, continuous functions, continuous bounded functions and
continuous functions vanishing at infinity. Together with the
supnorm $||\cdot||_\infty:=\sup_{x\in\R}|\cdot|$, $B_b(\R)$,
$C_b(\R)$ and $C_0(\R)$ are a Banach spaces.   Furthermore,
$(\chain{X},\{\mathbb{P}^{x}\}_{x\in\R})$,
$(\chain{X^{\alpha(x)}},\{\mathbb{P}^{x}\}_{x\in\R})$,
$(\chain{X^{(\alpha,\beta)}},\{\mathbb{P}^{x}\}_{x\in\R})$ and
$(\chain{X^{p}},\{\mathbb{P}^{x}\}_{x\in\R})$ will denote the
stable-like chains on $(\R,\mathcal{B}(\R))$ given by
(\ref{eq:1.1}), (\ref{eq:1.2}), (\ref{eq:1.3}) and (\ref{eq:1.4}),
respectively,  while $(\chain{Y},\{\mathbb{P}^{x}\}_{x\in\R})$ and
$(\process{Y},\{\mathbb{P}^{x}\}_{x\in\R})$ will denote an arbitrary
Markov chain and an arbitrary c\`adl\`ag strong Markov process on
$(\R,\mathcal{B}(\R))$ given by
 transition kernels $p(x,B)$ and  $p^{t}(x,B)$, for
$x\in \R$, $B\in \mathcal{B}(\R)$ and $t\in\R_+$, respectively.
Using the notation from \cite{tweedie}, we use the term \emph{Markov
model} and notation  $\{Y_t\}_{t\in\mathbb{T}}$, where $\mathbb{T}$
is either $\ZZ_+$ or $\R_+$, when the result holds regardless of the
time set involved. For $x\in\R$, $B\in \mathcal{B}(\R)$ and
$n\in\N$, let $p^{n}(x,B):=\mathbb{P}^{x}(Y_n\in B)$. For $x\in\R$
and $B\in \mathcal{B}(\R)$  we put
$\eta_B:=\sum_{n=0}^{\infty}1_{\{Y_n\in
                                                    B\}}$ or $\eta_B:=\int_{0}^{\infty}1_{\{Y_t\in
                                                    B\}}dt$, $\tau_B:=\inf \{n\geq0:
                                                    Y_n\in B\}$ or $\tau_B:=\inf \{t\geq0:
                                                    Y_t\in B\}$,
                                                    $Q(x,B):=\mathbb{P}^{x}(\eta_B=\infty)$,
                                                    $L(x,B):=\mathbb{P}^{x}(\tau_B<\infty)$ and
                                                    $U(x,B):=\mathbb{E}^{x}(\eta_B).$

\section{Preliminary and auxiliary results}

In this section we give some  preliminary and auxiliary results
needed for proving  the main results of this paper.
\begin{definition}
A  Markov model $\{Y_t\}_{t\in\mathbb{T}}$ on $(\R,\mathcal{B}(\R))$
is \emph{$\varphi$-irreducible} if there exists a probability
measure $\varphi(\cdot)$ on $\mathcal{B}(\R)$ such that, whenever
$\varphi(B)>0$, we have $U(x,B)>0$ for all $x\in\R$.
\end{definition}
Note that the stable-like chains $\chain{X}$ and
$\chain{X^{\alpha(x)}}$ (the chains given by (\ref{eq:1.1}) and
(\ref{eq:1.2}))
 are $\varphi-$irreducible for any probability measure $\varphi(\cdot)$ on
$\mathcal{B}(\R)$ which is absolutely continuous with respect to the
Lebesgue measure (see \cite[Proposition 2.1]{ja}).

In   \cite[Theorem 2.1]{tweedie} it is shown that the irreducibility
measure can always be  maximized. If $\{Y_t\}_{t\in\mathbb{T}}$ is a
$\varphi$-irreducible Markov model on $(\R,\mathcal{B}(\R))$, then
there exists
 a probability measure $\psi(\cdot)$ on $\mathcal{B}(\R)$  such that the model $\{Y_t\}_{t\in\mathbb{T}}$ is
 $\psi$-irreducible and $\bar{\varphi}\ll\psi$, for every irreducibility
 measure $\bar{\varphi}(\cdot)$ on $\mathcal{B}(\R)$ of the model
 $\{Y_t\}_{t\in\mathbb{T}}$. The measure $\psi(\cdot)$ is called the \emph{maximal
irreducibility measure} and from now on, when we refer to the
irreducibility measure we actually refer to the maximal
irreducibility measure. For the $\psi$-irreducible Markov model
$\{Y_t\}_{t\in\mathbb{T}}$ on $(\R,\mathcal{B}(\R))$ set
$\mathcal{B}^{+}(\R)=\{B\in\mathcal{B}(\R):\psi(B)>0\}.$ The maximal
irreducibility measure for the stable-like chains $\chain{X}$ and
$\chain{X^{\alpha(x)}}$  is equivalent, in absolutely continuous
sense, with the Lebesgue measure  (see \cite[Proposition 2.1]{ja}).

 Recall that a function
$f:\R\longrightarrow\R$ is called \emph{lower semicontinuous} if
$\liminf_{y\longrightarrow x}f(y)\geq f(x)$ holds for all $x\in\R$.
\begin{definition}Let $\{Y_t\}_{t\in\mathbb{T}}$ be a Markov
model on $(\R,\mathcal{B}(\R))$.
\begin{enumerate}
           \item [(i)] A set $B\in\mathcal{B}(\R)$ is \emph{uniformly transient}  if there exists  a finite constant $M\geq0$ such that
$U(x,B)\leq M$ holds for all $x\in\R$. The model
$\{Y_t\}_{t\in\mathbb{T}}$ is \emph{transient } if it is
           $\psi$-irreducible and
if there exists a countable cover of $\R$ with uniformly transient
sets.
\item [(ii)] A set $B\in\mathcal{B}(\R)$ is \emph{recurrent}  if $U(x,B)=\infty$ holds for all $x\in
 \R$. The model $\{Y_t\}_{t\in\mathbb{T}}$ is \emph{recurrent} if it  is $\psi$-irreducible and if every set $B\in \mathcal{B}^{+}(\R)$ is recurrent.
 \item [(iii)]A set $B\in\mathcal{B}(\R)$ is \emph{Harris
               recurrent}, or \emph{H-recurrent},  if $Q(x,B)=1$ holds for all $x\in
 \R$. The  model $\{Y_t\}_{t\in\mathbb{T}}$ is \emph{H-recurrent}
               if it is
$\psi$-irreducible and if every set $B\in \mathcal{B}^{+}(\R)$ is
H-recurrent.
\item [(iv)]The model $\{Y_t\}_{t\in\mathbb{T}}$ is called a \emph{T-model} if for some distribution $a(\cdot)$ on $\mathbb{T}$ there exists a kernel $T(x,B)$ with
$T(x,\R) > 0$ for all $x\in\R$, such that the function $x\longmapsto
T(x,B)$ is lower semicontinuous for all $B\in\mathcal{B}(\R)$, and
$$\int_{\mathbb{T}}p^{t}(x,B)a(dt)\geq T(x,B)$$ holds for all $x\in\R$ and all $B\in\mathcal{B}(\R).$
\end{enumerate}
\end{definition}
Let us remark that the H-recurrence property can be defined in the
equivalent way: The model $\{Y_t\}_{t\in\mathbb{T}}$ is H-recurrent
if it is $\psi$-irreducible and if $L(x,B) = 1$ holds for all
$x\in\R$ and all $B\in\mathcal{B}^{+}(\R)$ (see \cite [Theorem
2.4]{tweedie}). In general, recurrence and H-recurrence properties
are not equivalent (see \cite[Chapter 9]{tweedie}). Obviously,
H-recurrence implies recurrence. In the case of a Markov model which
is a $\lambda$-irreducible T-model, these two properties are
equivalent (see \cite[Proposition 5.3]{ja} and \cite[Theorem
4.2]{bjoern}).

In the following proposition, by assuming certain continuity
properties, we determine ``nice" sets for Markov models.
\begin{proposition}\label{p2.3} Let $\{Y_t\}_{t\in\mathbb{T}}$ be a
$\psi$-irreducible  Markov model, then:
\begin{enumerate}
\item [(i)] the  model $\{Y_t\}_{t\in\mathbb{T}}$ is either
recurrent or transient.
\end{enumerate}
In addition, if we assume that $\{Y_t\}_{t\in\mathbb{T}}$ is a
T-model, then:
\begin{enumerate}
\item [(ii)] the  model $\{Y_t\}_{t\in\mathbb{T}}$ is H-recurrent if and only if there
  exists a  H-recurrent compact set.
\item[(iii)] assume the following additional assumption in the continuous-time case: for every
  compact
set $C\in\mathcal{B}(\R)$ there exists a distribution $a_C(\cdot)$
on $\R_+$, such that \be\label{eq:2.1} \inf_{x\in
C}\int_0^{\infty}\mathbb{P}^{x}(X_t\in B)a_C(dt)
>0\ee holds for all $B\in\mathcal{B}^{+}(\R^{d})$. Then
the model $\{Y_t\}_{t\in\mathbb{T}}$ is transient  if and only if
  every compact set is uniformly transient.
\item [(iv)] under  the assumption (\ref{eq:2.1}) for the continuous-time case, the  model $\{Y_t\}_{t\in\mathbb{T}}$ is recurrent if and only if there
  exists a recurrent compact  set.
\end{enumerate}
\end{proposition}
\begin{proof}
\begin{itemize}
\item[(i)] The proof is given in \cite[Theorem 2.3]{tweedie}.
  \item [(ii)] The proof is given in \cite[Proposition 9.1.7]{meyn} and \cite[Theorem 3.3]{meyntweedie}.
  \item [(iii)] The proof for the discrete-time case is given in \cite[Theorems 8.3.5]{meyn}. If the process $\{Y_t\}_{t\geq0}$ is transient, then there
  exists at least one uniformly transient set $B\in\mathcal{B}^{+}(\R)$.
  By assumption (\ref{eq:2.1}), $$\delta_C:=\inf_{x\in
C}\int_0^{\infty}\mathbb{P}^{x}(X_t\in B)a_C(dt)
>0$$ holds for every compact set $C\in\mathcal{B}(\R)$. Using the Chapman-Kolmogorov equation we
have:
\begin{align*}U(x,B)&=\int_0^{\infty}U(x,B)a_C(dt)=
\int_0^{\infty}\int_0^{\infty}p^{s}(x,B)dsa_C(dt)\\
&\geq\int_{0}^{\infty}\int_0^{\infty}p^{s+t}(x,B)dsa_C(dt)=\int_{0}^{\infty}\int_0^{\infty}\int_{\R}p^{s}(x,dy)p^{t}(y,B)dsa_C(dt)\\
&\geq\delta_C\int_0^{\infty}\int_{C}p^{s}(x,dy)ds=\delta_CU(x,C).
\end{align*}
  \item [(iv)] The proof follows directly from (i) and (iii).
\end{itemize}
  \end{proof}

Now, we derive the recurrence and transience dichotomy by using
sample-paths properties of Markov models. Let $B\in\mathcal{B}(\R)$
be arbitrary and let $\mathbb{D}(\R)$ be the space of real-valued
c\`adl\`ag functions equipped with the Skorohod topology. In the
continuous-time case, define the \emph{set of recurrent paths} by:
$$R(B):=\{\omega\in\ \mathbb{D}(\R):\forall n\in\N,\ \exists t\geq n \ \textrm{such that} \ \omega(t)\in B\},$$ and the \emph{set of transient
paths}  by:
$$T(B):=\{\omega\in\ \mathbb{D}(\R):\exists s\geq0 \ \textrm{such that} \ \omega(t)\not\in B, \ \forall t\geq s
\}.$$ It is clear that $T(B)=R(B)^{c}$, and for any open set
$O\subseteq\R$, by the right continuity, $R(O)$ and $T(O)$ are
measurable (with respect to the Borel $\sigma$-algebra generated by
the Skorohod topology). In the discrete-time case, using the same
notation, we similarly define the \emph{set of recurrent paths} by:
$$R(B):=\{\omega\in\R^{\ZZ_+}:\forall n\in\N,\ \exists m\geq n \ \textrm{such that} \ \omega(m)\in B\},$$ and the \emph{set of transient
paths}  by:
$$T(B):=\{\omega\in\R^{\ZZ_+}:\exists m\geq0 \ \textrm{such that} \ \omega(n)\not\in B, \ \forall n\geq
m \}.$$ Clearly, $T(B)=R(B)^{c}$ and for any $B\in\mathcal{B}(\R)$,
$R(B)$ and  $T(B)$ are $\mathcal{B}(\R)^{\ZZ_+}$ measurable.

\begin{proposition}\label{p2.4} Let $\textbf{Y}=\{Y_t\}_{t\in\mathbb{T}}$  be a $\lambda$-irreducible
T-model, and  let us assume (\ref{eq:2.1})  holds  for the
continuous-time case. Then the following 0-1 property must be met:
$$\mathbb{P}^{x}\left(\eta_O=\infty\right)=0 \quad \textrm{for all} \ x\in\R \ \textrm{and all open bounded sets}\ O\subseteq\R $$ or
$$\mathbb{P}^{x}\left(\eta_O=\infty\right)=1 \quad \textrm{for all} \ x\in\R \ \textrm{and all open bounded sets}\quad
O\subseteq\R.$$ In particular, the model $\textbf{Y}$ is recurrent
if
             and only if $\mathbb{P}^{x}_\textbf{Y}(R(O))=1$ for all $x\in\R$ and
             all
             open bounded sets $O\subseteq\R$, and
              it is transient if
             and only if $\mathbb{P}^{x}_\textbf{Y}(T(O))=1$ for all $x\in\R$ and
             all
             open bounded sets $O\subseteq\R$.
\end{proposition}
\begin{proof} The 0-1 property  in the discrete-time case follows   from \cite
[Proposition 5.3]{ja} and \cite[Theorems 6.2.5 and 8.3.5]{meyn}. The
claim in the continuous-time case follows from Proposition
\ref{p2.3}, \cite[Theorem 5.1]{tweedie} and \cite[Theorem
4.2]{bjoern}. Now, the characterization  by sample paths
 easily follows from the 0-1 property  and  \cite[Theorem 3.3]{meyntweedie}.
\end{proof}

As already mentioned, the stable-like chains $\chain{X}$ and
$\chain{X^{\alpha(x)}}$ (the chains given by (\ref{eq:1.1}) and
(\ref{eq:1.2})) are $\lambda$-irreducible and, by \cite[Proposition
5.2]{ja}, the stable-like chain $\chain{X}$ is a T-model. In the
following proposition we give sufficient conditions for the
stable-like chain $\chain{X^{\alpha(x)}}$ to be a T-model.
\begin{proposition}\label{p2.5}Let $\alpha:\R\longrightarrow(0,2)$ and
$\gamma:\R\longrightarrow(0,\infty)$ be continuous functions. Then
the stable-like chain  $\chain{X^{\alpha(x)}}$  is a T-model. In
particular, $\chain{X^{\alpha(x)}}$ is H-recurrent if and only if it
is recurrent.
\end{proposition}
\begin{proof}Let us define $a(\cdot):=\delta_1(\cdot)$ and $T(x,B):=p(x,B)$ for $x\in\R$ and $B\in\mathcal{B}(\R)$. We  prove that the  function $x\longmapsto T(x,B)$ is lower semicontinuous
for every $B\in\mathcal{B}(\R).$ Let $x\in\R$ and
$B\in\mathcal{B}(\R)$ be arbitrary and such that
$\lambda(B)<\infty$. By the dominated convergence theorem and
continuity of the functions $\alpha(x)$ and $\gamma(x)$ we have
\begin{align*}\lim_{y\longrightarrow x}p(y,B)&=\lim_{y\longrightarrow
x}\int_Bf_{(\alpha(y),\gamma(y))}(z-y)dz\\&=(2\pi)^{-1}\lim_{y\longrightarrow
x}\int_{B}\int_{\R}\cos (\xi(z-y))
e^{-\gamma(y)|\xi|^{\alpha(y)}}d\xi dz\\&=\int_{B}\int_{\R}\cos
(\xi(z-x)) e^{-\gamma(x)|\xi|^{\alpha(x)}}d\xi
dz\\&=p(x,B).\end{align*} Let $B\in\mathcal{B}(\R)$ be arbitrary,
then, by Fatou's lemma, we have
\begin{align}\label{eq:2.2}\liminf_{y\longrightarrow
x}p(y,B)=&\liminf_{y\longrightarrow x}\sum_{n\in\ZZ}p(y,B\cap(
n,n+1]) \geq \sum_{n\in\ZZ}p(x,B\cap( n,n+1])=p(x,B).\end{align}
\end{proof}

Recall that a Markov model $\{Y_t\}_{t\in\mathbb{T}}$ is said to
satisfy the \emph{$C_b$-Feller property} if for all $f\in C_b(\R)$
and all $t\in\mathbb{T}$ the function $x\longmapsto\int_\R
p^{t}(x,dy)f(y)$ is in the space $C_b(\R)$. Furthermore, a Markov
model $\{Y_t\}_{t\in\mathbb{T}}$ is said to satisfy the \emph{strong
Feller property} if for all $f\in B_b(\R)$ and all
$t\in\mathbb{T}\setminus\{0\}$ the function $x\longmapsto\int_\R
p^{t}(x,dy)f(y)$ is in the space $C_b(\R)$. In \cite[Proposition
6.1.1]{meyn} it is shown that the $C_b$-Feller property
(respectively the strong Feller property) of a Markov model is
equivalent with the lower semicontinuity of the function
$x\longmapsto p^{t}(x,O)$ (respectively the function $x\longmapsto
p^{t}(x,B)$) for all open sets $O\subseteq\R$ (respectively all
Borel sets $B\subseteq\R$) and all $t\in\mathbb{T}\setminus\{0\}$.
Note that (\ref{eq:2.2}) reads that the stable-like chain
$\chain{X^{\alpha(x)}}$ satisfies the $C_b$-Feller property and the
strong Feller property.

Unfortunately, the stable-like chain $\chain{X^{(\alpha,\beta)}}$
(the chain given by (\ref{eq:1.3})) does not satisfy the
$C_b$-Feller property and the strong Feller property
($\liminf_{y\longrightarrow0} p(y,O)\geq p(0,O)$ does not hold for
some open sets $O\subseteq\R$). We introduce its ``continuous" and
in the recurrence and transience sense equivalent version: Let $k>0$
be arbitrary and let $\chain{\bar{X}^{(\alpha,\beta)}}$ be the
stable-like Markov chain defined by transition densities with
following characteristic functions
\begin{align*}\bar{\varphi}(x;\xi)=
                                                      \exp(-\bar{\gamma}(x)|\xi|^{\bar{\alpha}(x)}),\end{align*}
                                                      where
                                                      functions
                                                      $\bar{\alpha}:\R\longrightarrow(0,2)$
                                                      and
                                                      $\bar{\gamma}:\R\longrightarrow(0,\infty)$
                                                      are continuous
                                                      functions such
                                                      that
\begin{align*}\bar{\alpha}(x)=\left\{\begin{array}{cc}
                                                      \alpha, & x<-k \\
                                                      \beta, & x>k
                                                    \end{array}\right.   \quad \textrm{and}\quad  \bar{\gamma}(x)=\left\{\begin{array}{cc}
                                                      \gamma, & x<-k \\
                                                      \delta,&
                                                      x>k.
                                                    \end{array}\right.\end{align*}

\begin{proposition}\label{p2.6} The stable-like chain $\chain{\bar{X}^{(\alpha,\beta)}}$ is
recurrent if and only if the stable-like chain
$\chain{X^{(\alpha,\beta)}}$ is recurrent.
\end{proposition}
\begin{proof}
By \cite[Propostion 5.4]{ja}, it suffices to prove that condition
(C5) holds, i.e., that there exists $l
> 0$ such that for all compact sets $C\subseteq[-l, l]^{c}$ with
$\lambda(C) > 0,$ we have
$$\inf_{x\in[-k,k]}\int_{C-x}f_{(\bar{\alpha}(x),\bar{\gamma}(x))}(dy)>0.$$
Let us take $l=k$. Without loss of generality, let $C\subseteq(
k,\infty)$ be a compact set with $\lambda(C) > 0$. Then by symmetry
and bell-shaped property of densities
$f_{(\bar{\alpha}(x),\bar{\gamma}(x))}(y)$ (see \cite[Theorem
1]{gawronski}), we have
$$\inf_{x\in[-k,k]}\int_{C-x}f_{(\bar{\alpha}(x),\bar{\gamma}(x))}(y)dy\geq
\inf_{x\in[-k,k]}\int_{C+k}f_{(\bar{\alpha}(x),\bar{\gamma}(x))}(y)dy.$$
Let us assume that
$\inf_{x\in[-k,k]}\int_{C-x}f_{(\bar{\alpha}(x),\bar{\gamma}(x))}(y)dy=0$.
Then there exists a sequence $\{x_n\}_{n\in\N}\subseteq[-k,k]$, such
that $\lim_{n\longrightarrow\infty}x_n=x_0\in[-k,k]$ and
$$\lim_{n\longrightarrow\infty}\int_{C+k}f_{(\bar{\alpha}(x_n),\bar{\gamma}(x_n))}(y)dy=(2\pi)^{-1}\lim_{n\longrightarrow\infty}\int_{C+k}\int_{\R}\cos (\xi y) e^{-\bar{\gamma}(x_n)|\xi|^{\bar{\alpha}(x_n)}}d\xi dy=0.$$
Now, by the dominated convergence theorem and continuity of the
functions $\bar{\alpha}(x)$ and $\bar{\gamma}(x)$ we have
\begin{align*}0&=(2\pi)^{-1}\lim_{n\longrightarrow\infty}\int_{C+k}\int_{\R}\cos
(\xi y)e^{-\bar{\gamma}(x_n)|\xi|^{\bar{\alpha}(x_n)}}d\xi dy\\
&=(2\pi)^{-1}\int_{C+k}\int_{\R}\cos (\xi y)
e^{-\bar{\gamma}(x_0)|t|^{\bar{\alpha}(x_0)}}d\xi dy=
\int_{C+k}f_{(\bar{\alpha}(x_0),\bar{\gamma}(x_0))}(y)dy.\end{align*}
This is impossible since $\lambda(C)>0$.
\end{proof}

 For a Markov process
$\process{Y}$ we define a family of operators $\process{P}$  on
$B_b(\R)$ by $P_tf(x):= \mathbb{E}^{x}(f(Y_t))$. Since $\process{Y}$
is a Markov process, the family $\process{P}$ forms a
\emph{semigroup} of linear operators on
$(B_b(\R),||\cdot||_\infty)$, i.e., $P_t\circ P_s=P_{t+s}$ and
$P_0=I$. Furthermore, the semigroup $\process{P}$ is
\emph{contractive} ($||P_tf||_{\infty}\leq||f||_{\infty}$ for all
$f\in B_b(\R)$) and \emph{positivity preserving} ($P_tf\geq 0$
whenever $f\geq0$, $f\in B_b(\R)$).     The process $\process{Y}$ is
said to be a $C_0$-\emph{Feller process} if the semigroup
$\process{P}$ forms a \emph{Feller semigroup}. This means that:
\begin{enumerate}
  \item [(i)] the family $\process{P}$ is a semigroup of linear
operators on the space $C_0(\R)$;
  \item [(ii)] the family $\process{P}$ is \emph{strongly continuous}, i.e., $\lim_{t\longrightarrow0}||P_tf-f||_{\infty}=0$.
\end{enumerate}
The \emph{infinitesimal generator} $\mathcal{A}$ of the semigroup
$\process{P}$ is defined by
$$\mathcal{A}f:=
  \lim_{t\longrightarrow0}\frac{P_tf-f}{t}$$
on $\mathcal{D}_{\mathcal{A}}:=\{f\in B_b(\R):
\lim_{t\longrightarrow0}\frac{P_t f-f}{t} \quad \textrm{exists in
supnorm}\}$. If the set of smooth functions with compact support
$C_c^{\infty}(\R)$ is contained in $\mathcal{D}_{\mathcal{A}}$ and
$\mathcal{A}(C_c^{\infty}(\R))\subseteq C(\R)$, then
$\mathcal{A}|_{C_c^{\infty}(\R)}$ is a \emph{pseudo-differential
operator}, i.e., it can be written in the form
\begin{align}\label{eq:1}\mathcal{A}|_{C_c^{\infty}(\R)}f(x) = -\int_{\R}p(x,\xi)e^{ix\xi}
\widehat{f}(\xi) d\xi,\end{align}  where $\widehat{f}(\xi)=
(2\pi)^{-1} \int_{\R} e^{-ix\xi} f(x) dx$ is the Fourier transform
of $f(x)$ (see \cite[Theorem 3.4]{courege}). The function $p :
\R\times \R\longrightarrow \CC$ is called  the \emph{symbol} of the
pseudo-differential operator. It is measurable and locally bounded
in $(x,\xi)$ and continuous and negative definite as a function of
$\xi$. Hence, by \cite[Theorem 3.7.7]{jacobI}, $\xi\longmapsto
p(x,\xi)$ has for each $x$ the L\'{e}vy-Khinchine representation,
i.e., \begin{align}\label{eq:2}p(x,\xi) =a(x)- ib(x)\xi +
\frac{1}{2}c(x)\xi^{2} -
\int_{\R}\left(e^{iy\xi}-1-iy\xi1_{\{z:|z|\leq1\}}(y)\right)\nu(x,dy),\end{align}
where $a(x)\geq0$, $b(x)\in\R$ and $c(x)\geq0$ are Borel measurable
functions and $\nu(x,\cdot)$ is a Borel kernel on $\R\times
\mathcal{B}(\R)$, such that $\nu(x,\{0\})=0$ and $\int_{\R}(1\wedge
y^{2})\nu(x,dy)<\infty$ holds for all $x\in\R$. The quadruple
$(a(x),b(x),c(x),\nu(x,\cdot))$ is called the
\emph{L\'{e}vy-quadruple} of the pseudo-differential operator
$\mathcal{A}|_{C_c^{\infty}(\R)}$. In the following we  assume,
without loss of generality, that every Feller process  has
c\`adl\`ag paths (see \cite[Theorem III.2.7]{revuzyor}).

\begin{proposition}\label{p2.7}Let $a\neq0$ be arbitrary and let
                                                   $\process{N^{\kappa}}$
                                                   be the Poisson
                                                   process with
                                                   parameter
                                                   $\kappa>0$
                                                   independent of a
                                                     Markov chain
                                                   $\chain{Y}$ on $(\R,\mathcal{B}(\R))$. Then the process $\process{Y^{(a,\kappa)}}$, defined by $Y^{(a,\kappa)}_t:=aY_{N^{\kappa}_t}$, is \begin{enumerate}
                                                                    \item
                                                                    [(i)]
                                                                    a
                                                                    strong
                                                                    Markov
                                                                    process
                                                                    with
                                                                    the strongly
                                                                    continuous
                                                                    semigroup
                                                                    $\process{P^{(a,\kappa)}}$
                                                                    and
                                                                    the
                                                                    infinitesimal
                                                                    generator
                                                                    $$\mathcal{A}^{(a,\kappa)}f(x)=\kappa\int_{\R}(f(y)-f(x))p\left(a^{-1}x,a^{-1}dy\right)$$
                                                                    with
                                                                    the
                                                                    domain
                                                                    $\mathcal{D}_{\mathcal{A}^{(a,\kappa)}}=B_b(\R^{d});$

\item [(ii)] $\lambda$-irreducible and
                                                                   recurrent
                                                                   (respectively
                                                                   H-recurrent)
                                                                    if
                                                                    and
                                                                    only
                                                                    if
                                                                    the
                                                                    chain
                                                                    $\chain{Y}$
                                                                    is
                                                                    $\lambda$-irreducible
                                                                    and
                                                                    recurrent (respectively
                                                                   H-recurrent).
 \end{enumerate}
\end{proposition}
\begin{proof}
\begin{itemize}
               \item [(i)] First, note that if $\chain{Y}$ is a
               Markov chain with respect to the family of probability
               measures
               $\{\mathbb{P}^{x}\}_{x\in\R^{d}}$, then
               $\chain{aY}$ is a
               Markov chain with respect to the family of probability
               measures
               $\{\mathbb{Q}^{x}:=\mathbb{P}^{a^{-1}x}\}_{x\in\R}$.
Hence, the process $\process{Y^{(a,\kappa)}}$ is  a strong Markov
process. Clearly, its transition
               kernel is given by
               $$p^{t}(x,dy)=e^{-\kappa
               t}\sum_{n=0}^{\infty}\frac{(\kappa
               t)^{n}}{n!}p^{n}\left(a^{-1}x,a^{-1}dy\right).$$
               Now, the claim easily follows.

               \item [(ii)] The equivalence of $\lambda$-irreducibility and
               recurrence between the process
               $\process{Y^{(a,\kappa)}}$ and the  chain $\chain{Y}$ easily follows from the definition and the fact that the exponential distribution has finite all moments. In the
               case of H-recurrence  we have
              $$L^{(a,\kappa)}(x,B)=\mathbb{Q}^{x}(\tau^{(a,\kappa)}_B<\infty)=\mathbb{P}^{a^{-1}x}(\tau_{a^{-1}B}<\infty)=L(a^{-1}x,a^{-1}B).$$
              Hence, the process $\process{Y^{(a,\kappa)}}$ is H-recurrent  if and only if the  chain $\chain{Y}$ is
               H-recurrent.
\end{itemize}
\end{proof}

It is natural to expect that if the functions
$\alpha:\R\longrightarrow(0,2)$ and
$\gamma:\R\longrightarrow(0,\infty)$ are continuous, the  process
$\process{Y^{\alpha(x)}}:=\{aX^{\alpha(x)}_{N^{\kappa}_t}\}_{t\geq0}$
is a $C_0$-Feller process. We need the following lemma.

\begin{lemma}\label{l2.8}Let $0<\varepsilon<2$ and $C>0$ be arbitrary, and let $\alpha:\R\longrightarrow(\varepsilon,2)$ and
$\gamma:\R\longrightarrow(0,C)$ be arbitrary
 functions. Furthermore, let
$\{f_{(\alpha(x),\gamma(x))}\}_{x\in\R}$ be a family of
S$\alpha(x)$S densities given by the following characteristic
functions $\varphi(x;\xi)=
                                                       \exp(-\gamma(x)|\xi|^{\alpha(x)}).$
                                                       Then the following uniformity condition
                                                       holds
\begin{align}\label{eq:2.3}\lim_{b\longrightarrow\infty}\sup_{x\in\R}\int_b^{\infty}f_{(\alpha(x),\gamma(x))}(y)dy=0.\end{align}
Moreover,
\begin{align*}\lim_{|y|\longrightarrow\infty}\sup_{\{x\in\R:\,\alpha(x)<1\}}\left|f_{(\alpha(x),\gamma(x))}(y)\frac{|y|^{\alpha(x)+1}}{c(x)}-1\right|=0,\end{align*}
 where
$$c(x)=\frac{\gamma(x)}{\pi}\Gamma(\alpha(x)+1)\sin\left(\frac{\pi\alpha(x)}{2}\right).$$
\end{lemma}
\begin{proof}
Let $0<\rho<\varepsilon$ be arbitrary and let $\{Z_x\}_{x\in \R}$ be
a family of random variables with S$\alpha(x)$S distributions with
densities $\{f_{(\alpha(x),\gamma(x))}\}_{x\in \R}$. Then we have
$$\sup_{x\in \R}\int_b^{\infty}f_{(\alpha(x),\gamma(x))}(y)dy=\sup_{x\in \R}\mathbb{P}(Z_x\geq b)\leq\sup_{x\in
\R}\mathbb{P}(|Z_x|\geq b)\leq\frac{1}{b^{\rho}}\sup_{x\in
\R}\mathbb{E}|Z_x|^{\rho}.$$ Since $\sup_{x\in
\R}\mathbb{E}|Z_x|^{\rho}$ is finite (see \cite[page 163]{sato}),
the first claim easily follows.

To prove the second part of lemma  we use \cite[Theorem
2.4.2]{zolotarev}. Since $\alpha(x)<1$, for $|y|\geq1$ we have
\begin{align*}&\left|f_{(\alpha(x),\gamma(x))}(y)\frac{|y|^{\alpha(x)+1}}{c(x)}-1\right|\\
&=\frac{1}{\Gamma(\alpha(x)+1)\sin\left(\frac{\pi\alpha(x)}{2}\right)}\left|\sum_{n=2}^{\infty}(-1)^{n+1}\frac{\Gamma(n\alpha(x)+1)}{n!}\sin\left(\frac{n\pi\alpha(x)}{2}\right)\left(\frac{\gamma(x)}{|y|^{\alpha(x)}}\right)^{n-1}\right|\\
&\leq\frac{1}{\Gamma(\varepsilon+1)\sin\left(\frac{\pi\varepsilon}{2}\right)}\sum_{n=1}^{\infty}\left(\frac{C}{|y|^{\varepsilon}}\right)^{n}.\end{align*}
Now, by taking $\sup_{\{x\in\R:\,\alpha(x)<1\}}$ and letting
$|y|\longrightarrow\infty$, we get the desired result.
\end{proof}

\begin{proposition}\label{p2.9}Let $0<\varepsilon<2$ and $C>0$ be arbitrary,  let $\alpha:\R\longrightarrow(\varepsilon,2)$ and
$\gamma:\R\longrightarrow(0,C)$ be  continuous functions.
Furthermore, let $a\neq0$ be arbitrary and let
                                                   $\process{N^{\kappa}}$
                                                   be the Poisson
                                                   process with
                                                   parameter
                                                   $\kappa>0$
                                                   independent of
                                                   the  stable-like chain
$\chain{X^{\alpha(x)}}$ (the chain given by (\ref{eq:1.2})). Then
the process
$\process{Y^{\alpha(x)}}:=\{aX^{\alpha(x)}_{N^{\kappa}_t}\}_{t\geq0}$
is
\begin{enumerate}
  \item [(i)] a $C_0$-Feller process with
                                                                    the
                                                                    symbol
                                                                    $$p(x,\xi)=a^{-1}\kappa\left(1-\int_{\R}e^{i\xi y}f_{(\alpha(a^{-1}x),\gamma(a^{-1}x))}(a^{-1}y)dy\right)$$
                                                                    and the L\'{e}vy
quadruple $(0,0,0,a^{-1}\kappa
f_{(\alpha(a^{-1}x),\gamma(a^{-1}x))}(a^{-1}y)dy),$ and it
 satisfies the  $C_b$-Feller property and the
strong Feller property;
  \item [(ii)] a
                                                                    T-model.
\end{enumerate}
\end{proposition}
\begin{proof} By Proposition \ref{p2.7}, the semigroup of
the process $\process{Y^{\alpha(x)}}$ is given by
$$P^{\alpha(x)}_tf(x)=e^{-\kappa
               t}\sum_{n=0}^{\infty}\frac{(\kappa
               t)^{n}}{n!}\int_{\R}p^{n}\left(a^{-1}x,a^{-1}dy\right)f(y),$$ for $f\in
               B_b(\R)$, and the generator
\begin{align}\label{eq:3}\mathcal{A}^{\alpha(x)}f(x)=a^{-1}\kappa\int_{\R}(f(y+x)-f(x))f_{(\alpha(a^{-1}x),\gamma(a^{-1}x))}(a^{-1}y)dy\end{align}
 with the domain $\mathcal{D}_{\mathcal{A}^{\alpha(x)}}=B_b(\R).$
Furthermore, it is shown that the semigroup is strongly continuous.
\begin{enumerate}
  \item [(i)] The $C_b$-Feller property easily follows from (\ref{eq:2.2}) and Fatou's lemma.
Now, let us  show that $P^{\alpha(x)}_t(C_0(\R))\subseteq C_0(\R)$
for all $t\in\R_+$. For $f\in C_0(\R)$, by the $C_b$-Feller
property, $P^{\alpha(x)}_tf\in C_b(\R)$ for all $t\in\R_+$. Next we
show that $P^{\alpha(x)}_tf(x)$ vanishes at infinity for all $f\in
C_0(\R)$ and all $t\in\R_+$. Let $f\in C_0(\R)$ and $\epsilon>0$ be
arbitrary such that $||f||_\infty\leq M$, for some $M\geq0$. Since
$C_c(\R)$ is dense in $(C_0(\R),||\cdot||_\infty)$, there exists
$f_\epsilon\in C_c(\R)$ such that
$||f-f_\varepsilon||_\infty<\epsilon$. We have
\begin{align*}\left|\int_\R p(a^{-1}x,a^{-1}dy)f(y)\right|&\leq\int_\R
p(a^{-1}x,a^{-1}dy)|f(y)|<\int_\R
p(a^{-1}x,a^{-1}dy)|f_\epsilon(y)|+\epsilon\\&=a^{-1}\int_{\rm{supp}\,\it{f}_\epsilon-x}f_{\alpha(a^{-1}x),\gamma(a^{-1}x))}(a^{-1}y)|
f_\epsilon(y+x)|dy+\epsilon\\&\leq
a^{-1}(M+\epsilon)\int_{\rm{supp}\,\it{f}_\epsilon-x}f_{\alpha(a^{-1}x),\gamma(a^{-1}x))}(a^{-1}y)dy+\epsilon.\end{align*}
Since $\rm{supp}\,\it{f}_\epsilon$ is a compact set, by applying
Lemma \ref{l2.8}, the function $x\longmapsto\int_\R
p(a^{-1}x,a^{-1}dy)f(y)$ is a $C_0(\R)$ function. Thus, by the
dominated convergence theorem we have the claim,  i.e., the process
$\process{Y^{\alpha(x)}}$ is a $C_0$-Feller process.

The second part of the proposition easily follows from the relations
(\ref{eq:1}), (\ref{eq:2}) and (\ref{eq:3}), and the strong Feller
property follows from \cite[Theorem 1.1]{rene-wang-feller}
 \item [(ii)] The claim follows from \cite[Theorem 7.1]{tweedie}.
\end{enumerate}
\end{proof}

Let us  recall the notion of characteristics of a semimartingale
(see \cite{jacod} or \cite{schnurr}). Let
$(\Omega,\mathcal{F},\process{F},\mathbb{P},\process{S})$,
$\process{S}$ in the sequel, be a semimatingale and let
$h:\R\longrightarrow\R$ be a truncation function (i.e., a continuous
bounded function such that $h(x)=x$ in a neighborhood of the
origin).
 We  define two processes
$$\check{S}(h)_t:=\sum_{s\leq t}(\Delta S_s-h(\Delta S_s))\quad
\textrm{and} \quad S(h)_t:=S_t-\check{S}(h)_t,$$ where the process
$\process{\Delta S}$ is defined by $\Delta S_t:=S_t-S_{t-}$ and
$\Delta S_0:=S_0$. The process $\process{S(h)}$ is a special
semimartingale. Hence, it admits the unique decomposition
\begin{align}\label{eq:2.4}S(h)_t=S_0+M(h)_t+B(h)_t,\end{align} where $\process{M(h)}$ is a local
martingale and $\process{B(h)}$ is a predictable process of bounded
variation.

\begin{definition}
 Let $\process{S}$  be a
semimartingale and let $h:\R\longrightarrow\R$ be the truncation
function. Furthermore, let $\process{B(h)}$  be the predictable
process
 of bounded variation appearing in (\ref{eq:2.4}),  let $N(\omega,ds,dy)$ be the
compensator of the jump measure
$$\mu(\omega,ds,dy)=\sum_{s:\Delta S_s(\omega)\neq 0}\delta_{(s,\Delta S_s(\omega))}(ds,dy)$$ of the process
$\process{S}$ and let $\process{C}$ be the quadratic co-variation
process for $\process{S^{c}}$ (continuous martingale part of
$\process{S}$), i.e.,
$$C_t=\langle S^{c}_t,S^{c}_t\rangle.$$  Then $(B,C,N)$ is called
the \emph{characteristics} of the semimartingale $\process{S}$
 (relative to $h(x)$). If we put $\tilde{C}(h)_t:=\langle
 M(h)_t,M(h)_t\rangle$, where $\process{M(h)}$ is the local martingale
 appearing in (\ref{eq:2.4}), then $(B,\tilde{C},N)$ is called the \emph{modified
 characteristics} of the semimartingale $\process{S}$ (relative to $h(x)$).
\end{definition}

\begin{proposition}\label{p2.11}
Let $a\neq0$ be arbitrary and let $\{f_x\}_{x\in\R}$ be a family of
probability densities on the real line such that $x\longmapsto
f_x(y)$ is a Borel measurable function for all $y\in\R$. Let
$\chain{Y}$ be a Markov chain  on $(\R,\mathcal{B}(\R))$, with
respect to the filtration $\chain{\mathcal{F}}$, given by the
 transition kernel $p(x,dy):=f_x(y-x)dy$. Furthermore, let
$\process{Y^{(a,\kappa)}}$ be the process defined by
$Y^{\kappa}_t:=aY_{N^{\kappa}_t},$
                                                   where
                                                   $\process{N^{\kappa}}$
                                                   is the Poisson
                                                   process, with
respect to the  filtration $\process{\mathcal{G}}$, with
                                                   parameter
                                                   $\kappa>0$
                                                   independent of
                                                   the  chain
                                                   $\chain{Y}.$
Then the process $\process{Y^{(a,\kappa)}}$  is a semimartingale
with respect to the filtration
$\{\sigma\{\mathcal{F}_\infty\cup\mathcal{G}_t\}\}_{t\geq0}$, where
$\mathcal{F}_{\infty}=\bigcup_{n=0}^{\infty}\mathcal{F}_n$, and its
characteristics and the modified characteristics, relative to the
truncation function $h(x)$, are given by:
\begin{align*}
B^{(a,\kappa)}_t&=a^{-1}\kappa\int_0^{t}\int_{\R}h(y)f_{Y_{N^{\kappa}_{s}}}(a^{-1}y)dyds,\\
C^{(a,\kappa)}_t&=0,\\
\tilde{C}^{(a,\kappa)}_t&= a^{-1}\kappa\int_0^{t}\int_{\R}h^{2}(y)f_{Y_{N^{\kappa}_{s}}}(a^{-1}y)dyds\quad and \\
N^{(a,\kappa)}(ds,dy)&=a^{-1}\kappa
f_{Y_{N^{\kappa}_{s}}}(a^{-1}y)dyds.
\end{align*}
\end{proposition}
\begin{proof}
Clearly, the process $\process{Y^{(a,\kappa)}}$ is a semimartingale.
By Proposition \ref{p2.7},  the infinitesimal generator of the
process $\process{Y^{(a,\kappa)}}$ is given by
$\mathcal{A}^{(a,\kappa)}f(x)=a^{-1}\kappa\int_{\R}(f(y+x)-f(x))f_{a^{-1}x}(a^{-1}y)dy,$
 $f\in B_b(\R)$. Furthermore, by \cite[Poposition IV.1.7]{ethier}, for every $f\in B_b(\R)$ the
process
$$M^{f}_t:=f(Y^{(a,\kappa)}_t)-f(Y^{(a,\kappa)}_0)-\int_0^{t}\mathcal{A}^{(a,\kappa)}f(Y^{(a,\kappa)}_{s-})ds$$ is a
martingale. Let $h(x)$ be the truncation function and let $f\in
C_b^{1}(\R)$. Then $\process{M^{f}}$ can be rewritten in the
following form
\begin{align*}
M^{f}_t&=f(Y^{(a,\kappa)}_t)-f(Y^{(a,\kappa)}_0)-a^{-1}\kappa\int_0^{t}\int_{\R}\left(f(y+Y^{(a,\kappa)}_{s-})-f(Y^{(a,\kappa)}_{s-})\right)f_{Y_{N^{\kappa}_{s-}}}(a^{-1}y)dyds\\
&=f(Y^{(a,\kappa)}_t)-f(Y^{(a,\kappa)}_0)-a^{-1}\kappa\int_0^{t}\int_{\R}f'(Y^{(a,\kappa)}_{s-})h(y)f_{Y_{N^{\kappa}_{s-}}}(a^{-1}y)dyds\\
&\ \ \ \
-a^{-1}\kappa\int_0^{t}\int_{\R}\left(f(y+Y^{(a,\kappa)}_{s-})-f(Y^{(a,\kappa)}_{s-})-f'(Y^{(a,\kappa)}_{s-})h(y)\right)f_{Y_{N^{\kappa}_{s-}}}(a^{-1}y)dyds.\end{align*}
Now, from \cite[Proposition II.2.17 and Theorem II.2.42]{jacod}, the
claim follows.
\end{proof}

We refer the reader to \cite{jacod, schilling, schnurr} for more
details about characteristics of a semimartingale and connection
with Feller processes.

As we know, the recurrence property of  S$\alpha$S random walk,
given by the characteristic function
$\varphi(\xi)=\exp(-\gamma|\xi|^{\alpha})$, depends only on the
index of stability $\alpha\in(0,2]$ and it does not depend on the
scaling constant $\gamma\in(0,\infty)$. In the following proposition
we show that this is also the case with the stable-like chain
$\chain{X^{\alpha(x)}}$ (the chain given by (\ref{eq:1.2})).

\begin{proposition}\label{p2.13}
 Let $\chain{X^{\alpha(x)}}$ be the
stable-like chain defined in Proposition \ref{p2.9}. Furthermore,
let $c>0$ be arbitrary and let $\chain{X^{(\alpha(x),c)}}$ be the
stable-like chain which we get by replacing the scaling function
$\gamma(x)$ by the scaling function $c\gamma(x)$.  Then the
stable-like chain $\chain{X^{(\alpha(x),c)}}$ is recurrent if and
only if the stable-like chain $\chain{X^{\alpha(x)}}$ is recurrent.
\end{proposition}
\begin{proof}
Let
                                                   $\process{N^{1}}$
                                                   be the Poisson
                                                   process with
                                                   parameter
                                                   $1$
                                                   independent of
                                                   the stable-like
                                                   chain
                                                   $\chain{X^{(\alpha(x),c)}}$. Let us define the process $\textbf{X}^{c}=\process{X^{c}}$  by
                                                   $X^{c}_t:=X^{(\alpha(x),c)}_{N_t}.$
                                                   By Proposition \ref{p2.11}, the process
                                                   $\textbf{X}^{c}$
                                                   has
                                                   the modified
                                                   characteristics (relative to the truncation function $h(x)$) given by: \begin{align*}
B^{c}_t&=\int_0^{t}\int_{\R\setminus\{0\}}h(y)f_{(\alpha(X^{c}_{s}),c\gamma(X^{c}_{s}))}(y)dyds,\\
C^{c}_t&=0,\\
\tilde{C}^{c}_t&= \int_0^{t}\int_{\R\setminus\{0\}}h^{2}(y)f_{(\alpha(X^{c}_{s}),c\gamma(X^{c}_{s}))}(y)dyds\quad and \\
N^{c}(ds,dy)&= f_{(\alpha(X^{c}_{s}),c\gamma(X^{c}_{s}))}(y)dyds.
\end{align*}

Let $c_0>0$ be arbitrary and fixed and let us show that
$$\textbf{X}^{c}\stackrel{\hbox{\scriptsize d}}{\longrightarrow}
\textbf{X}^{c_0},\ \ \textrm{when}\ \ c\longrightarrow c_0,$$ where
$\stackrel{\hbox{\scriptsize d}}{\longrightarrow}$ denotes the
convergence in the space of c\`adl\`ag functions equipped with the
Skorohod topology.
 We only have to check assumptions 4.3, 4.7, 4.9, 4.10, 4.11 and 4.12
 from
\cite[Theorem IX.4.8]{jacod}. Assumptions 4.3, 4.7, 4.10 and 4.12
can be easily verified  by use of \cite[Theorem 3.3.5]{durrett},
continuity assumption of the functions $\alpha(x)$ and $\gamma(x)$,
the dominated convergence theorem and Propositions \ref{p2.9} and
\ref{p2.11}, while assumption 4.9 follows from Lemma \ref{l2.8}. To
verify assumption 4.11 we have to show that
$$\lim_{c\longrightarrow
c_0}\sup_{x\in[a,b]}\left|\int_{\R}g(y)\left(f_{(\alpha(x),c_0\gamma(x))}(y)-f_{(\alpha(x),c\gamma(x))}(y)\right)dy\right|=0$$
holds for all $g\in C_b(\R)$ and all $[a,b]\subseteq\R.$ If that
would not be the case, then there would exist $g\in C_b(\R)$,
$[a,b]\subseteq\R$, $\delta>0$ and sequences $\{c_n\}_{n\in\N}$ and
$\{x_n\}_{n\in\N}\subseteq[a,b]$ with limits $c_0$ and
$x_0\in[a,b],$ respectively, such that
\be\label{eq:2.8}\left|\int_{\R}g(y)\left(f_{(\alpha(x_n),c_0\gamma(x_n))}(y)-f_{(\alpha(x_n),c_n\gamma(x_n))}(y)\right)dy\right|>\delta\ee
holds for all $n\in\N$. Let $M\geq0$ be such that
$||g(x)||_{\infty}\leq M$ and let $R>0$ be arbitrary. We have
\begin{eqnarray*}&&\left|\int_{\R}g(y)\left(f_{(\alpha(x_n),c_0\gamma(x_n))}(y)-f_{(\alpha(x_n),c_n\gamma(x_n))}(y)\right)dy\right|\nonumber\\
&&\begin{split}\leq&\left|\int_{-R}^{R}g(y)\left(f_{(\alpha(x_n),c_0\gamma(x_n))}(y)-f_{(\alpha(x_n),c_n\gamma(x_n))}(y)\right)dy\right|\\&+\left|\int_{|y|\geq
R}g(y)\left(f_{(\alpha(x_n),c_0\gamma(x_n))}(y)-f_{(\alpha(x_n),c_n\gamma(x_n))}(y)\right)dy\right|.\end{split}\end{eqnarray*}
From continuity of the functions $\alpha(x)$ and $\gamma(x)$ and
from \cite[Corollary 1.2.4]{ushakov}, we have
\begin{align*}\lim_{n\longrightarrow\infty}\left|\int_{-R}^{R}g(y)\left(f_{(\alpha(x_n),c_0\gamma(x_n))}(y)-f_{(\alpha(x_n),c_n\gamma(x_n))}(y)\right)dy\right|=0.\end{align*}
Furthermore, by Lemma \ref{l2.8} we have
\begin{align*}&\lim_{R\longrightarrow\infty}\sup_{n\in\N}\left|\int_{|y|\geq
R}g(y)\left(f_{(\alpha(x_n),c_0\gamma(x_n))}(y)-f_{(\alpha(x_n),c_n\gamma(x_n))}(y)\right)dy\right|\nonumber\\
&\leq M\lim_{R\longrightarrow\infty}\sup_{n\in\N}\int_{|y|\geq
R}f_{(\alpha(x_n),c_0\gamma(x_n))}(y)+M\lim_{R\longrightarrow\infty}\sup_{n\in\N}\int_{|y|\geq
R}f_{(\alpha(x_n),c_n\gamma(x_n))}(y)=0.\end{align*} Hence,
$$\lim_{n\longrightarrow\infty}\left|\int_{\R}g(y)\left(f_{(\alpha(x_n),c_0\gamma(x_n))}(y)-f_{(\alpha(x_n),c_n\gamma(x_n))}(y)\right)dy\right|=0,$$
what is in contradiction with (\ref{eq:2.8}). The locally uniform
convergence of the other two characteristics easily follows from
\cite[Theorem 3.3.5]{durrett}, the continuity assumption of the
functions $\alpha(x)$ and $\gamma(x)$ and the dominated convergence
theorem.

Let $x\in\R$ be arbitrary and let $O\subseteq\R$ be an arbitrary
open bounded set. Since the  process
             $\textbf{X}^{c_0}$ satisfies the $C_b$-Feller property, by
             \cite[Lemmas 2 and 3]{franke}, we have $\mathbb{P}^{x}_{\textbf{X}^{c_0}}(\partial R(O))=0.$  Here $\partial A$ denotes the boundary of the set $A$.  Therefore, by   \cite[Theorem 2.1]{bil}, we have
             \begin{align*}\lim_{c\longrightarrow c_0}\mathbb{P}^{x}_{\textbf{X}^{c}}(R(O))=\mathbb{P}^{x}_{\textbf{X}^{c_0}}(R(O))\end{align*}
             for all $x\in\R$ and for all open bounded  sets
             $O\subseteq\R$. Hence, for all $x\in\R$ and all
             open bounded sets $O\subseteq\R$, the function $$c\longmapsto
             \mathbb{P}^{x}_{\textbf{X}^{c}}(R(O))$$ is a continuous function on $(0,\infty)$.
             Note that  (\ref{eq:2.1}) is satisfied if for the distribution
             $a_C(\cdot)$ we
             take $a_C(\cdot):=\delta_{t_0}(\cdot)$, where $t_0>0$ is
             arbitrary.
             Since the processes $\textbf{X}^{c}$ are $\lambda$-irreducible T-models, by   Proposition \ref{p2.4}, $\mathbb{P}^{x}_{\textbf{X}^{c}}(R(O))=1$
            for all
             $c\in(0,\infty)$, all $x\in\R$ and all open bounded
             sets
             $O\subseteq\R$, or $\mathbb{P}^{x}_{\textbf{X}^{c}}(R(O))=0$
            for all
             $c\in(0,\infty)$, all $x\in\R$ and all open bounded
             sets
             $O\subseteq\R$. This means, again by  Proposition \ref{p2.4},
             that all processes $\textbf{Y}^{c}$, $c\in(0,\infty)$, are either recurrent
             or transient at the same time. Now,
             by Proposition \ref{p2.7},  the desired result follows.
\end{proof}

\section{Proof of  Theorem 1.1}
In this section we give a proof of  Theorem \ref{tm1.1}. Let the
function $p : \R\times \R\longrightarrow \CC$ be given by
$p(x,\xi)=\gamma(x)|\xi|^{\alpha(x)}$, for some functions
$\alpha:\R\longrightarrow(0,2)$ and
$\gamma:\R\longrightarrow(0,\infty)$. In \cite{bass} it is shown
that if the functions $\alpha(x)$ and $\gamma(x)$
satisfy:\begin{enumerate}
                                                  \item [(i)] $0<\inf_{x\in\R}\alpha(x)\leq\sup_{x\in\R}\alpha(x)<2$
                                                  and
                                                  $0<\inf_{x\in\R}\gamma(x)\leq\sup_{x\in\R}\gamma(x)<\infty$,
                                                  \item [(ii)]
                                                  $\beta(z)=o(1/|\ln(z)|),$
                                                  when
                                                  $z\longrightarrow0$,
                                                  where
                                                  $\beta(z):=\sup_{|x-y|\leq
                                                  z}|\alpha(x)-\alpha(y)|$,
                                                  \item [(iii)]
                                                  $\int_{0}^{1}\frac{\beta(z)}{z}dz<\infty,$
                                                  i.e., the function
                                                  $\alpha(x)$ is
                                                  Dini continuous and
                                                  \item [(iv)] $\gamma\in C(\R)$,
                                                \end{enumerate}
then the function (symbol) $p(x,\xi)=\gamma(x)|\xi|^{\alpha(x)}$
defines  $C_b$-Feller process on $(\R,\mathcal{B}(\R))$ called a
\emph{stable-like process}.
 Note that if the function $\alpha(x)$ is
Lipschitz continuous, i.e., if there exists $L>0$, such that
$|\alpha(x)-\alpha(y)|\leq L|x-y|$ holds for all $x,y\in\R$, then it
is also Dini continuous and  condition (ii) is satisfied. Write
$C^{1}_b(\R)$ for the set of bounded continuously differentiable
functions with bounded derivative. Clearly, $\alpha\in C_b^{1}(\R)$
implies
 Lipschitz continuity of $\alpha(x)$. Furthermore, by \cite[Theorems
 1.1 and 3.3]{rene-wang-feller}, $\alpha,\gamma\in C_b^{1}(\R)$ imply that the corresponding stable-like process is a $C_0$-Feller process and it satisfies the
 $C_b$-Feller property and the strong Feller property.

\begin{proof}[Proof of Theorem \ref{tm1.1}]
Let $k>0$ be arbitrary and let
$\textbf{X}^{(\alpha,\beta)}=\process{X^{(\alpha,\beta)}}$ be the
stable-like process on $\R$ which corresponds to the symbol
$p(x,\xi)=\gamma(x)|\xi|^{\alpha(x)}$, where the functions
$\alpha,\gamma \in C_b^{1}(\R)$ are such that
\begin{align*}\alpha(x)=\left\{\begin{array}{cc}
                                                      \alpha, & x<-k \\
                                                      \beta, & x>k
                                                    \end{array}\right.   \quad \textrm{and}\quad  \gamma(x)=\left\{\begin{array}{cc}
                                                      \gamma, & x<-k \\
                                                      \delta,&
                                                      x>k.
                                                    \end{array}\right.\end{align*}
By \cite[Theorem 5.1]{kolokoltsov}, transition kernel
$\mathbb{P}^{x}(X^{(\alpha,\beta)}_t\in dy)$ is absolutely
continuous with respect to the Lebesgue measure, and by
\cite[Theorem 3.3]{rene-wang-feller},
$\mathbb{P}^{x}(X^{(\alpha,\beta)}_t\in B)>0$ holds for all
$x\in\R$, all $t\in\R_+$  and all $B\in\mathcal{B}(\R)$ with
$\lambda(B)>0$. Therefore, the stable-like
                                                    process
                                                    $\textbf{X}^{(\alpha,\beta)}$ is
                                                    $\lambda-$irreducible
                                                    and, by
                                                    \cite[Theorem 7.1]{tweedie}, it
                                                    is a
                                                    T-model.
                                                    Hence, from
                                                      \cite[Theorem 4.2]{bjoern},
                                                    H-recurrence and
                                                    recurrence
                                                    properties of
                                                    the stable-like
                                                     process $\textbf{X}^{(\alpha,\beta)}$
                                                    are equivalent.
                                                    Furthermore,
                                                     from
                                                    \cite[Corollary 5.5]{bjoern}, the stable-like
                                                    process $\textbf{X}^{(\alpha,\beta)}$
                                                    is recurrent if
                                                    and only if
                                                    $\alpha+\beta\geq2$.

                                                    By \cite{bjoernschilling},  the stable-like process
                                                    $\textbf{X}^{(\alpha,\beta)}$ can be approximated by a
                                                    sequence of
                                                    Markov chains,
                                                    i.e., for a sequence of Markov chains
                                                    $\{X^{(m)}_n\}_{n\geq0}$,
                                                    $m\in\N$,
on $(\R,\mathcal{B}(\R))$ given by a sequence of transition kernels
$p_m(x,dy)$, $m\in\N$, such that
$$\int_{\R}e^{i\xi y}p_m(x,dy)=e^{i\xi x-\frac{1}{m}\eta(x,\xi)}=e^{i\xi x-\frac{\gamma(x)}{m}|\xi|^{\alpha(x)}},$$
we have that
$$\textbf{X}^{(m)}\stackrel{\hbox{\scriptsize
d}}{\longrightarrow} \textbf{X}^{(\alpha,\beta)},\ \ \textrm{as}\ \
m\longrightarrow\infty,$$ where $\textbf{X}^{(m)}=\{X^{(m)}_{\lfloor
mt\rfloor}\}_{t\geq0}$. Again, $\stackrel{\hbox{\scriptsize
d}}{\longrightarrow}$ denotes  convergence in distribution in the
space of c\`adl\`ag functions equipped with the Skorohod topology.
By Proposition \ref{p2.13}, the chains $\{X^{(m)}_n\}_{n\geq0}$,
                                                    $m\in\N$, are
                                                    either recurrent
                                                    or transient at
                                                    the same time.
                                                    Hence, the rest of
                                                     proof is
                                                    devoted to prove
                                                    that this
                                                    dichotomy is
                                                    equivalent with
                                                    the recurrence-transience dichotomy of
                                                    the stable like
                                                    process $\textbf{X}^{(\alpha,\beta)}$.

Since the stable-like process
             $\textbf{X}^{(\alpha,\beta)}$ is a $C_b$-Feller process, by
             \cite[Lemmas 2 and 3]{franke} we have $\mathbb{P}^{x}_{\textbf{X}^{(\alpha,\beta)}}(\partial R(O))=0$ for
             all
             $x\in\R$ and all
             open bounded  sets $O\subseteq\R$.  Therefore, by  \cite[Theorem 2.1]{bil}, we have
             \begin{align}\label{eq:3.1}\lim_{m\longrightarrow\infty}\mathbb{P}^{x}_{\textbf{X}^{(m)}}(R(O))=\mathbb{P}^{x}_{\textbf{X}^{(\alpha,\beta)}}(R(O))\end{align}
             for all $x\in\R$ and for all open bounded  sets $O\subseteq\R$.

Let us assume that $\alpha+\beta\geq2$. Hence, the stable-like
process $\textbf{X}^{(\alpha,\beta)}$ is recurrent. Note that
assumption (\ref{eq:2.1}) follows if for  the distribution
$a_C(\cdot)$ we take $a_C(\cdot):=\delta_{t_0}(\cdot),$ where
$t_0>0$ is arbitrary, and apply the strong Feller property. Hence,
by the 0-1 property (Proposition \ref{p2.4})
$\mathbb{P}^{x}_{\textbf{X}^{(\alpha,\beta)}}(R(O))=1$ holds for all
$x\in\R$ and all open bounded sets $O\subseteq\R$.  From
(\ref{eq:3.1}), for any starting point $x\in\R$ and any open bounded
set $O\subseteq\R$ there exists $m_0\geq1$ such that
$\mathbb{P}^{x}_{\textbf{X}^{(m_0)}}(R(O))>0$, i.e.,
$\mathbb{P}^{x}\left(\sum_{n=0}^{\infty}1_{\{X^{(m_0)}_n\in
O\}}=\infty\right)>0.$ But, since the stable-like chain
$\{X^{(m_0)}_n\}_{n\geq0}$ is $\lambda-$irreducible T-model, by 0-1
property,
$$\mathbb{P}^{x}\left(\sum_{n=0}^{\infty}1_{\{X^{(m_0)}_n\in
O\}}=\infty\right)=1$$ holds for all $x\in\R$ and all open bounded
sets $O\subseteq\R$, i.e., the stable like chain
$\{X^{(m_0)}_n\}_{n\geq0}$ is recurrent. Now, by applying
Proposition \ref{p2.13},  all  stable-like chains
$\{X^{(m)}_n\}_{n\geq0}$,
                                                    $m\in\N$, are
                                                    recurrent. Therefore, since
$\{\bar{X}^{(\alpha,\beta)}_n\}_{n\geq0}\stackrel{\hbox{\scriptsize
d}}{=}\{X^{(1)}_n\}_{n\geq0}$ (recall that the stable-like chain
$\{\bar{X}^{(\alpha,\beta)}_n\}_{n\geq0}$ is defined in Proposition
\ref{p2.6}), by Propositions \ref{p2.6}  the stable-like chain
$\chain{X^{(\alpha,\beta)}}$ is recurrent.

Let us now show that the recurrence property  of the stable-like
chain $\chain{X^{(\alpha,\beta)}}$ implies $\alpha+\beta\geq2.$ Let
us assume that this is not the case, i.e., let us assume that
$\alpha+\beta<2$. Hence, the stable-like process
$\textbf{X}^{(\alpha,\beta)}$ is transient, i.e.,
$\mathbb{P}^{x}_{\textbf{X}^{(\alpha,\beta)}}(T(O))=1$ holds for all
$x\in\R$ and all
             open bounded sets $O\subseteq\R$. Now, by (\ref{eq:3.1}), we have
$$\lim_{m\longrightarrow\infty}\mathbb{P}^{x}_{\textbf{X}^{(m)}}(T(O))=
\mathbb{P}^{x}_{\textbf{X}^{(\alpha,\beta)}}(T(O))=1.$$
 Hence, for any starting point $x\in\R$ and any open bounded set $O\subseteq\R$, there exists $m_0\geq1$ such that
$\mathbb{P}^{x}_{\textbf{X}^{(m_0)}}(T(O))>0$. Therefore,
$\mathbb{P}^{x}\left(\sum_{n=0}^{\infty}1_{\{X^{(m_0)}_n\in
O\}}=\infty\right)<1.$ Again, by the 0-1 property, we have
$$\mathbb{P}^{x}\left(\sum_{n=0}^{\infty}1_{\{X^{(m_0)}_n\in
O\}}=\infty\right)=0$$ for all $x\in\R$ and all open bounded sets
$O\subseteq\R$. Hence, the stable like chain
$\{X^{(m_0)}\}_{n\geq0}$ is transient. Therefore, by Proposition
\ref{p2.13},  all the stable-like chains $\{X^{(m)}_n\}_{n\geq0}$,
                                                    $m\in\N$, are
                                                    transient.
Since
$\{\bar{X}^{(\alpha,\beta)}_n\}_{n\geq0}\stackrel{\hbox{\scriptsize
d}}{=}\{X^{(1)}_n\}_{n\geq0}$, by  Proposition \ref{p2.6}, the
stable-like chain $\chain{X^{(\alpha,\beta)}}$ is also transient.
But this is in contradiction with recurrence assumption of the
stable-like chain $\chain{X^{(\alpha,\beta)}}$. Hence, we have
proved the desired result.
\end{proof}

\section{Proof of  Theorem 1.2}
In this section we give a proof of  Theorem \ref{tm1.2}. Recall that
the functions $x\longmapsto f_x$, $\alpha(x)$ and $c(x)$ are
$\tau$-periodic, the function $(x,y)\longmapsto f_x(y)$ is
continuous and strictly positive and $\alpha(x)$ and $c(x)$ are
Borel measurable. Let us put $\Lambda:=\tau\ZZ$ and let
$\Pi_{\Lambda} : \R\longrightarrow \R/\Lambda$ be the covering map.
We denote by $\chain{X^{\Lambda p}}$ the process on $\R/\Lambda$
obtained by projection of the stable-like chain $\chain{X^{p}}$ (the
chain given by (\ref{eq:1.4})) with respect to $\Pi_{\Lambda}(x)$.
By \cite[Proposition 3.8.8]{vasili-book}, the process
$\chain{X^{\Lambda p}}$ is a
 Markov chain on $\R/\Lambda$ with transition density
function
$$p^{\Lambda}(x,y)=\sum_{k\in\Lambda}p(z_x,z_y+k)=\sum_{k\in\Lambda}f_{z_x}(z_y-z_x+k)$$ for all $x,y\in\R/\Lambda$, where    $z_x$ and $z_y$ are arbitrary points  in
$\Pi^{-1}_{\Lambda}(\{x\})$ and $\Pi^{-1}_{\Lambda}(\{y\})$,
respectively. Furthermore, by \cite[Theorem
III.3.1]{benso-lions-book}, the chain $\chain{X^{\Lambda p}}$
possesses an invariant measure $\pi(\cdot)$, with
$\pi(\R/\Lambda)<\infty$, and there exist constants $C>0$ and $c>0$,
such that for all $\tau$-periodic functions $f\in B_b(\R)$  we have
$$\int_{\R/\Lambda}f(z_x)\pi(dx)=0\quad\Longrightarrow\quad
\left|\left|\int_{\R}p_n(\cdot,dy)f(y)\right|\right|_{\infty}\leq
C||f||_{\infty}e^{-cn}\quad\textrm{for all}\quad n\in\N.$$ Since
$\pi(\R/\Lambda)<\infty$, without loss of generality, we assume that
$\pi(\R/\Lambda)=1.$ Following the ideas from the proof of
\cite[Theorem 1]{franke}, we give the proof of Theorem \ref{tm1.2}.

\begin{proof}[Proof of  Theorem \ref{tm1.2}]
 Let $\chain{X^{\Lambda p}}$ be as above.  Let us suppose that the set
$\{x\in\R:\alpha(x)=\alpha_0:=\inf_{x\in\R}\alpha(x)\}$ has positive
Lebesgue measure. By $\lambda$-irreducibility of the stable-like
chain $\chain{X^{p}}$, this is equivalent with
$\pi(\Pi_{\Lambda}(\{x\in\R:\alpha(x)=\alpha_0:=\inf_{x\in\R}\alpha(x)\}))>0$.
Indeed, since $\pi(\cdot)$ is the invariant measure of the chain
$\chain{X^{\Lambda p}}$,
$$\int_{\R/\Lambda}p^{\Lambda}(x,B)\pi(dx)=\pi(B)$$ holds for all
$B\in\mathcal{B}(\R/\Lambda)$, where $\mathcal{B}(\R/\Lambda)$
denotes the Borel $\sigma$-algebra with respect to the quotient
topology. Let us put
$A:=\{x\in\R:\alpha(x)=\alpha_0:=\inf_{x\in\R}\alpha(x)\}$ and
$B:=\Pi_{\Lambda}(A).$ We have
\begin{align*}\pi(B)&=\int_{\R/\Lambda}p^{\Lambda}(x,B)\pi(dx)=\int_{\R/\Lambda}p(z_x,\Pi_{\Lambda}^{-1}(B))\pi(dx)=\int_{\R/\Lambda}p(z_x,A)\pi(dx).\end{align*}
Now, if $\lambda(A)>0$, then $p(z_x, A)>0$ for all $z_{x}\in\R$.
Therefore, $\pi(B)>0$ as well. On the other hand, if $\lambda(A)=0$,
then $p(z_x,A)=0$ for all $z_{x}\in \R$. Hence, $\pi(B)=0$.

In the sequel (because of $\tau$-periodicity) we use the
abbreviation $\alpha(x)$ and $c(x)$, for $\alpha\left(z_x\right)$
and $c\left(z_x\right)$, where $x\in\R/\Lambda$ and
$z_x\in\Pi^{-1}_{\Lambda}(\{x\})$ are arbitrary.

Let $\process{N^{1}}$ be the Poisson process with parameter $1$
independent of the periodic stable-like chain $\chain{X^{p}}$ and
let us define a $\lambda$-irreducible  Markov process
$\textbf{Y}^{p}:=\{X^{p}_{N_t}\}_{t\geq0}.$ By Proposition
\ref{p2.7}, the semigroup of the process $\textbf{Y}^{p}$ is given
by
$$P_tf(x)=e^{-t}\sum_{n=0}^{\infty}\frac{t^{n}}{n!}\int_{\R}p^{n}(x,dy)f(y)$$
for $f\in B_b(\R)$ and $t\in\R_+$. Hence, for every $\tau$-periodic
function $f\in B_b(\R)$ we have
\be\label{eq:4.1}\left|\left|P_tf\right|\right|_{\infty}\leq
C||f||_{\infty}e^{-t}\sum_{n=0}^{\infty}\frac{t^{n}}{n!}e^{-cn}=C||f||_{\infty}e^{-t(1-e^{-c})}.\ee

Let us define the sequence of   semimartingales
$\textbf{Y}^{p}_{n}:=\{n^{-\frac{1}{\alpha_0}}X^{p}_{N_{nt}}\}$,
$n\in\N$. Now, we prove  that the sequence of processes
$\textbf{Y}^{p}_{n}$, $n\in\N$, converges in distribution to a
symmetric $\alpha_0$-stable L\'{e}vy process
$\textbf{L}=\process{L}$ with
 the modified characteristics (relative to
the truncation function $h(x)$)
\begin{align*}
B^{0}_t&=\Theta t\int_{\R}(h(y)-y1_{|y|\leq1})\frac{dy}{|y|^{\alpha_0+1}},\\
\tilde{C}^{0}_t&=\Theta t\int_{\R}h^{2}(y)\frac{dy}{|y|^{\alpha_0+1}} \quad \textrm{and} \\
N^{0}(ds,dy)&=\Theta\frac{dyds}{|y|^{\alpha_0+1}},
\end{align*} where $\Theta:=\int_{\R/\Lambda}1_{\{\alpha(x)=\alpha_0\}}c(x)\pi(dx)$ (see \cite[Theorem 3.5]{schilling}). Without loss of generality, we take all the processes $\textbf{Y}^{p}_{n}$,
$n\in\N$, and $\textbf{L}$ to be defined on the same probability
spaces $(\Omega,\mathcal{F},\{\mathbb{P}^{x}\}_{x\in\R})$. In order
to prove this convergence, by \cite[Theorem VIII.2.17]{jacod} it
suffices to show that initial distributions of $\textbf{Y}^{p}_{n}$
converge to initial distribution  of $\textbf{L}$ (what is trivially
satisfied) and the modified characteristics
$(B^{n},\widetilde{C}^{n},N^{n})$ of the processes
$\textbf{Y}^{p}_{n}$, $n\in\N$, converge in probability to the
modified characteristics $(B^{0},\tilde{C}^{0},N^{0})$, when
$n\longrightarrow\infty$. By Proposition \ref{p2.11}, the modified
characteristics $(B^{n},\tilde{C}^{n},N^{n})$ of the process
$\textbf{Y}^{p}_{n}$ are given by
\begin{align*}
B^{n}_t&=n^{1+\frac{1}{\alpha_0}}\int_0^{t}\int_{\R}h(y)f_{Y^{p}_{ns}}\left(n^{\frac{1}{\alpha_0}}y\right)dyds,\\
\tilde{C}^{n}_t&= n^{1+\frac{1}{\alpha_0}}\int_0^{t}\int_{\R}h^{2}(y)f_{Y^{p}_{ns}}\left(n^{\frac{1}{\alpha_0}}y\right)dyds\quad and \\
N^{n}(ds,dy)&=n^{1+\frac{1}{\alpha_0}}
f_{X^{p}_{ns}}\left(n^{\frac{1}{\alpha_0}}y\right)dyds.
\end{align*}Note  that (PC4), (\ref{eq:1.5}) and
$\lambda(\{x\in\R:\alpha(x)=\alpha_0:=\inf_{x\in\R}\alpha(x)\})>0$,
i.e.,
$\pi(\Pi_{\Lambda}(\{x\in\R:\alpha(x)=\alpha_0:=\inf_{x\in\R}\alpha(x)\}))>0$,
imply $0<\Theta<\infty,$ therefore the above $\alpha_0$-stable
L\'{e}vy process
 characteristics are well defined.

 Recall that
for a Borel measurable function $g:\R\longrightarrow\R$ and a random measure $\mu(\omega,ds,dx)$ on $\mathcal{B}(\R_+)\times\mathcal{B}(\R)$, the $\ast$-product is defined by $$ g\ast\mu_t(\omega):=\left\{\begin{array}{cc} \int_{[0,t]\times\R}g(x)\mu(\omega,ds,dx), & \int_{[0,t]\times\R}|g(x)|\mu(\omega,ds,dx)<\infty \\
                                                     \infty,
                                                      & \textrm{otherwise},
                                                    \end{array}\right.$$
 (see \cite[Definition II.1.3]{jacod} for details).  Let
$g\in C_b(\R)$ vanish in a neighborhood of the origin. We have
\begin{align}g\ast N^{n}_t=&\int_0^{t}\int_{\R}g(y)N^{n}(ds,dy)\nonumber\\
=&\int_0^{t}\int_{\R}g(y)n^{1+\frac{1}{\alpha_0}}
f_{Y^{p}_{ns}}\left(n^{\frac{1}{\alpha_0}}y\right)dyds\nonumber\\
=&\int_0^{t}\int_{\R}g\left(n^{\frac{1}{\alpha(Y^{p}_{ns})}-\frac{1}{\alpha_0}}y\right)n^{1+\frac{1}{\alpha(Y^{p}_{ns})}}
f_{Y^{p}_{ns}}\left(n^{\frac{1}{\alpha(Y^{p}_{ns})}}y\right)dyds\nonumber\\
=&\int_0^{t}\int_{\R}1_{\{\alpha(X^{p}_{ns})=\alpha_0\}}g\left(n^{\frac{1}{\alpha(Y^{p}_{ns})}-\frac{1}{\alpha_0}}y\right)n^{1+\frac{1}{\alpha(Y^{p}_{ns})}}
f_{Y^{p}_{ns}}\left(n^{\frac{1}{\alpha(Y^{p}_{ns})}}y\right)dyds\nonumber\\
&+\int_0^{t}\int_{\R}1_{\{\alpha(Y^{p}_{ns})>\alpha_0\}}g\left(n^{\frac{1}{\alpha(Y^{p}_{ns})}-\frac{1}{\alpha_0}}y\right)n^{1+\frac{1}{\alpha(Y^{p}_{ns})}}
f_{Y^{p}_{ns}}\left(n^{\frac{1}{\alpha(Y^{p}_{ns})}}y\right)dyds\nonumber\\
=\label{eq:4.2}&\int_0^{t}\int_{\R}1_{\{\alpha(Y^{p}_{ns})=\alpha_0\}}g\left(y\right)\frac{c(Y^{p}_{ns})}{|y|^{\alpha_0+1}}dyds\\
&\label{eq:4.3}+\int_0^{t}\int_{\R}1_{\{\alpha(Y^{p}_{ns})=\alpha_0\}}g\left(y\right)\left(n^{1+\frac{1}{\alpha_0}}
f_{Y^{p}_{ns}}\left(n^{\frac{1}{\alpha_0}}y\right)-\frac{c(Y^{p}_{ns})}{|y|^{\alpha_0+1}}\right)dyds\\
&\label{eq:4.4}+\int_0^{t}\int_{\R}1_{\{\alpha(Y^{p}_{ns})>\alpha_0\}}g\left(y\right)n^{1-\frac{\alpha(Y^{p}_{ns})}{\alpha_0}}\frac{c(Y^{p}_{ns})}{|y|^{\alpha(Y^{p}_{ns})+1}}dyds\\
&\label{eq:4.5}+\int_0^{t}\int_{\R}1_{\{\alpha(Y^{p}_{ns})>\alpha_0\}}g\left(y\right)\left(n^{1+\frac{1}{\alpha_0}}
f_{Y^{p}_{ns}}\left(n^{\frac{1}{\alpha_0}}y\right)-n^{1-\frac{\alpha(Y^{p}_{ns})}{\alpha_0}}\frac{c(Y^{p}_{ns})}{|y|^{\alpha(Y^{p}_{ns})+1}}\right)dyds.
\end{align}
Let $0<\varepsilon<1$ be arbitrary. Then, by (PC3), there exists
$y_{\varepsilon}\geq1$, such that
\be\label{eq:4.6}(1-\varepsilon)\frac{c(x)}{|y|^{\alpha(x)+1}}<f_x(y)<(1+\varepsilon)\frac{c(x)}{|y|^{\alpha(x)+1}}\ee
holds for all $|y|\geq y_{\varepsilon}$ and all $x\in\R$. Since the
function $g(x)$ vanishes in a neighborhood of the origin, by
(\ref{eq:4.6}) and the dominated convergence theorem, (\ref{eq:4.3})
and (\ref{eq:4.5}) converge to $0$, $\mathbb{P}^{x}$-a.s., when
$n\longrightarrow\infty.$ Let us  prove that (\ref{eq:4.4})
converges in $L^{2}(\Omega,\mathbb{P}^{x})$ to $0$,
 when $n\longrightarrow\infty$. We define
$$U_n(z):=\int_{\R}g(y)\left(1_{\{\alpha(z)>\alpha_0\}}n^{1-\frac{\alpha(z)}{\alpha_0}}\frac{c(z)}{|y|^{\alpha(z)+1}}-\int_{\R/\Lambda}1_{\{\alpha(x)>\alpha_0\}}n^{1-\frac{\alpha(x)}{\alpha_0}}\frac{c(x)}{|y|^{\alpha(x)+1}}\pi(dx)\right)dy.$$
By $\tau$-periodicity of functions $\alpha(x)$ and $c(x)$, the
function $U_n(z)$ is $\tau$-periodic and
$$\int_{\R/\Lambda}U_n(z)\pi(dz)=0.$$
Using integration by parts formula, Markov property and
(\ref{eq:4.1}), we have
\begin{align*}\mathbb{E}^{x}\left[\left(\int_0^{t}U_n(Y^{p}_{ns})\right)^{2}\right]&=2\int_0^{t}\int_0^{s}\mathbb{E}^{x}[U_n(Y^{p}_{ns})U_n(Y^{p}_{nr})]drds\nonumber\\
&=2\int_0^{t}\int_0^{s}\mathbb{E}^{x}[\mathbb{E}^{x}[U_n(Y^{p}_{ns})|\mathcal{F}_{nr}]U_n(Y^{p}_{nr})]drds\nonumber\\&=2\int_0^{t}\int_0^{s}\mathbb{E}^{x}[P_{n(s-r)}U_n(Y^{p}_{nr})U_n(Y^{p}_{nr})]drds\end{align*}
\begin{align}\label{eq:4.7}=2\int_0^{t}\int_0^{s}Ce^{-n(1-e^{-c})(s-r)}||U_n||^{2}_\infty=\frac{2C||U_n||^{2}_\infty}{n(1-e^{-c})}\int_0^{t}(1-e^{-n(1-e^{-c})s})ds\leq\frac{C||U_n||^{2}_{\infty}}{
n(1-e^{-c})}.\end{align}  Note that, by (\ref{eq:1.5}),
$$|U_n(z)|\leq\sup_{z\in\R}\left(\int_{\R}|g(y)|\frac{c(z)}{|y|^{\alpha(z)+1}}dy+\int_{\R}|g(y)|\int_{\R/\Lambda}\frac{c(x)}{|y|^{\alpha(x)+1}}\pi(dx)dy\right)<\infty,$$ i.e., $||U_n||_{\infty}$ remains bounded as $n$ grows. Hence
$$\lim_{n\longrightarrow\infty}\mathbb{E}^{x}\left[\left(\int_0^{t}U_n(Y^{p}_{ns})\right)^{2}\right]=0.$$
Furthermore,
\begin{align}\label{eq:4.8}&\left(\mathbb{E}^{x}\left[\left(\int_0^{t}\int_{\R}1_{\{\alpha(Y^{p}_{ns})>\alpha_0\}}g\left(y\right)n^{1-\frac{\alpha(Y^{p}_{ns})}{\alpha_0}}\frac{c(Y^{p}_{ns})}{|y|^{\alpha(Y^{p}_{ns})+1}}dyds\right)^{2}\right]\right)^{\frac{1}{2}}\nonumber\\
&\leq
\left(\mathbb{E}^{x}\left[\left(\int_0^{t}U_n(Y^{p}_{ns})\right)^{2}\right]\right)^{\frac{1}{2}}\nonumber\\
&  \quad
+\left(\mathbb{E}^{x}\left[\left(\int_0^{t}\int_{\R}\int_{\R/\Lambda}1_{\{\alpha(x)>\alpha_0\}}n^{1-\frac{\alpha(x)}{\alpha_0}}g(y)\frac{c(x)}{|y|^{\alpha(x)+1}}\pi(dx)dy\right)^{2}\right]\right)^{\frac{1}{2}}.\end{align}
By the dominated convergence theorem, (\ref{eq:4.8}) converges to
zero, when $n\longrightarrow\infty$, i.e.,  (\ref{eq:4.4}) converges
in $L^{2}(\Omega,\mathbb{P}^{x})$ to $0$, when
$n\longrightarrow\infty.$ Now, let us  prove that (\ref{eq:4.2})
converges in $L^{2}(\Omega,\mathbb{P}^{x})$ to
$$g\ast N_t^{0}=t\int_{\R}\int_{\R/\Lambda}1_{\{\alpha(x)=\alpha_0\}}g(y)\frac{c(x)}{|y|^{\alpha_0+1}}\pi(dx)dy,$$
when $n\longrightarrow\infty.$ We define
$$U(z):=\int_{\R}g(y)\left(1_{\{\alpha(z)=\alpha_0\}}\frac{c(z)}{|y|^{\alpha_0+1}}-\int_{\R/\Lambda}1_{\{\alpha(x)=\alpha_0\}}\frac{c(x)}{|y|^{\alpha_0+1}}\pi(dx)\right)dy.$$
By $\tau$-periodicity of functions $\alpha(x)$ and $c(x)$, the
function $U(z)$ is $\tau$-periodic and
$$\int_{\R/\Lambda}U(z)\pi(dz)=0.$$ Hence, in the same way as for  (\ref{eq:4.4}), it can be shown that  $g\ast N^{n}_t$ converges in probability to
$g\ast N_t^{0}$. In  the same way one can prove that $B_t^{n}$
converges in probability to $B_t^{0},$ when
$n\longrightarrow\infty.$

At the end, let us show that $\widetilde{C}_t^{n}$ converges in
probability to $\tilde{C}_t^{0}$, when $n\longrightarrow\infty$.
Recall that the truncation function $h(x)$ is a bounded Borel
measurable function satisfying $h(x)=x$ in a neighborhood of the
origin. Let $\delta>0$ be small enough and such that $h(x)=x$ for
all $x\in(-\delta,\delta).$ We have
\begin{align}\widetilde{C}_t^{n}=&\int_0^{t}\int_{\R}h^{2}(y)n^{1+\frac{1}{\alpha_0}}
f_{Y^{p}_{ns}}\left(n^{\frac{1}{\alpha_0}}y\right)dyds\nonumber\\
=&\int_0^{t}\int_{\R}1_{\{\alpha(Y^{p}_{ns})=\alpha_0\}}h^{2}\left(y\right)n^{1+\frac{1}{\alpha_0}}
f_{Y^{p}_{ns}}\left(n^{\frac{1}{\alpha_0}}y\right)dyds\nonumber\\
&+\int_0^{t}\int_{\R}1_{\{\alpha(Y^{p}_{ns})>\alpha_0\}}h^{2}\left(y\right)n^{1+\frac{1}{\alpha_0}}
f_{Y^{p}_{ns}}\left(n^{\frac{1}{\alpha_0}}y\right)dyds\nonumber\\
\label{eq:4.9}=&\int_0^{t}\int_{\R}1_{\{\alpha(Y^{p}_{ns})=\alpha_0\}}h^{2}\left(y\right)\frac{c(Y^{p}_{ns})}{|y|^{\alpha_0+1}}dyds\\
\label{eq:4.10}&+\int_0^{t}\int_{(-\delta,\delta)^{c}}1_{\{\alpha(Y^{p}_{ns})=\alpha_0\}}h^{2}\left(y\right)\left(n^{1+\frac{1}{\alpha_0}}
f_{Y^{p}_{ns}}\left(n^{\frac{1}{\alpha_0}}y\right)-\frac{c(Y^{p}_{ns})}{|y|^{\alpha_0+1}}\right)dyds\\
\label{eq:4.11}&+\int_0^{t}\int_{(-\delta,\delta)}1_{\{\alpha(Y^{p}_{ns})=\alpha_0\}}y^{2}\left(n^{1+\frac{1}{\alpha_0}}
f_{Y^{p}_{ns}}\left(n^{\frac{1}{\alpha_0}}y\right)-\frac{c(Y^{p}_{ns})}{|y|^{\alpha_0+1}}\right)dyds\\
\label{eq:4.12}&+\int_0^{t}\int_{\R}1_{\{\alpha(Y^{p}_{ns})>\alpha_0\}}h^{2}\left(y\right)n^{1-\frac{\alpha(Y^{p}_{ns})}{\alpha_0}}
\frac{c(Y^{p}_{ns})}{|y|^{\alpha(Y^{p}_{ns})+1}}dyds\\
\label{eq:4.13}&+\int_0^{t}\int_{(-\delta,\delta)^{c}}1_{\{\alpha(Y^{p}_{ns})>\alpha_0\}}h^{2}\left(y\right)\left(n^{1+\frac{1}{\alpha_0}}
f_{Y^{p}_{ns}}\left(n^{\frac{1}{\alpha_0}}y\right)-n^{1-\frac{\alpha(Y^{p}_{ns})}{\alpha_0}}
\frac{c(Y^{p}_{ns})}{|y|^{\alpha(Y^{p}_{ns})+1}}\right)dyds\\
\label{eq:4.14}&+\int_0^{t}\int_{(-\delta,\delta)}1_{\{\alpha(Y^{p}_{ns})>\alpha_0\}}y^{2}\left(n^{1+\frac{1}{\alpha_0}}
f_{Y^{p}_{ns}}\left(n^{\frac{1}{\alpha_0}}y\right)-n^{1-\frac{\alpha(Y^{p}_{ns})}{\alpha_0}}
\frac{c(Y^{p}_{ns})}{|y|^{\alpha(Y^{p}_{ns})+1}}\right)dyds.\end{align}
 By (\ref{eq:4.6}) and the dominated
convergence theorem, (\ref{eq:4.10}) and (\ref{eq:4.13})  converge
to $0$ $\mathbb{P}^{x}$-a.s., when $n\longrightarrow\infty.$ Let us
prove that (\ref{eq:4.11})
 converges to $0$ $\mathbb{P}^{x}$-a.s., when $n\longrightarrow\infty$ and
$\delta\longrightarrow0$, respectively. By using (\ref{eq:4.6}), we
have
\begin{eqnarray*}&&\int_0^{t}\int_{(-\delta,\delta)}1_{\{\alpha(Y^{p}_{ns})=\alpha_0\}}y^{2}\left(n^{1+\frac{1}{\alpha_0}}
f_{Y^{p}_{ns}}\left(n^{\frac{1}{\alpha_0}}y\right)-\frac{c(Y^{p}_{ns})}{|y|^{\alpha_0+1}}\right)dyds\\&&\begin{split}
=&\int_0^{t}\int_{(-n^{\frac{1}{\alpha_0}}\delta,n^{\frac{1}{\alpha_0}}\delta)}1_{\{\alpha(Y^{p}_{ns})=\alpha_0\}}y^{2}n^{1-\frac{2}{\alpha_0}}
f_{Y^{p}_{ns}}\left(y\right)dyds\\
&-\int_0^{t}\int_{(-\delta,\delta)}1_{\{\alpha(Y^{p}_{ns})=\alpha_0\}}|y|^{1-\alpha_0}c(Y^{p}_{ns})dyds\\
=&\int_0^{t}\int_{(-y_{\varepsilon},y_{\varepsilon})}1_{\{\alpha(Y^{p}_{ns})=\alpha_0\}}y^{2}n^{1-\frac{2}{\alpha_0}}
f_{Y^{p}_{ns}}\left(y\right)dyds\\
&+\int_0^{t}\int_{(-n^{\frac{1}{\alpha_0}}\delta,-y_{\varepsilon})\cup(
y_{\varepsilon},n^{\frac{1}{\alpha_0}}\delta)}1_{\{\alpha(Y^{p}_{ns})=\alpha_0\}}y^{2}n^{1-\frac{2}{\alpha_0}}
f_{Y^{p}_{ns}}\left(y\right)dyds\\
&+\frac{2}{2-\alpha_0}\delta^{2-\alpha_0}\int_0^{t}1_{\{\alpha(Y^{p}_{ns})=\alpha_0\}}c(Y^{p}_{ns})ds\\
\leq&
n^{1-\frac{2}{\alpha_0}}\int_0^{t}\int_{(-y_{\varepsilon},y_{\varepsilon})}1_{\{\alpha(Y^{p}_{ns})=\alpha_0\}}y^{2}
f_{Y^{p}_{ns}}\left(y\right)dyds\\
&+(1+\varepsilon)n^{1-\frac{2}{\alpha_0}}\int_0^{t}\int_{(-n^{\frac{1}{\alpha_0}}\delta,-y_{\varepsilon})\cup(
y_{\varepsilon},n^{\frac{1}{\alpha_0}}\delta)}1_{\{\alpha(Y^{(n)}_{s})=\alpha_0\}}y^{2}\frac{c(Y^{p}_{ns})}{|y|^{\alpha_0+1}}
dyds\\
&+\frac{2}{2-\alpha_0}\delta^{2-\alpha_0}\int_0^{t}1_{\{\alpha(Y^{p}_{ns})=\alpha_0\}}c(Y^{p}_{ns})ds
\end{split}
\end{eqnarray*}
\begin{align*}=&n^{1-\frac{2}{\alpha_0}}\int_0^{t}\int_{(-y_{\varepsilon},y_{\varepsilon})}1_{\{\alpha(Y^{p}_{ns})=\alpha_0\}}y^{2}
f_{Y^{p}_{ns}}\left(y\right)dyds\\
&+(1+\varepsilon)\frac{2}{2-\alpha_0}\delta^{2-\alpha_0}\int_0^{t}
1_{\{\alpha(Y^{p}_{ns})=\alpha_0\}}c(Y^{p}_{ns})ds \\
&-(1+\varepsilon)n^{1-\frac{2}{\alpha_0}}\frac{2}{2-\alpha_0}y_{\varepsilon}^{2-\alpha_0}\int_0^{t}
1_{\{\alpha(Y^{p}_{ns})=\alpha_0\}}c(Y^{p}_{ns})ds\\
&+\frac{2}{2-\alpha_0}\delta^{2-\alpha_0}\int_0^{t}1_{\{\alpha(Y^{p}_{ns})=\alpha_0\}}c(Y^{p}_{ns})ds.
\end{align*}  Now, by (\ref{eq:1.5}) and the dominated convergence
theorem, we have
$$\lim_{\delta\longrightarrow0}\lim_{n\longrightarrow\infty}\int_0^{t}\int_{(-\delta,\delta)}1_{\{\alpha(Y^{p}_{ns})=\alpha_0\}}y^{2}\left(n^{1+\frac{1}{\alpha_0}}
f_{Y^{p}_{ns}}\left(n^{\frac{1}{\alpha_0}}y\right)-\frac{c(Y^{p}_{ns})}{|y|^{\alpha_0+1}}\right)dyds=0
\quad \mathbb{P}^{x}\textrm{-a.s.}$$ In completely the same way one
can prove that (\ref{eq:4.14}) converges to 0 $\mathbb{P}^{x}$-a.s.,
when $n\longrightarrow\infty$ and $\delta\longrightarrow0$,
respectively. In order to prove that (\ref{eq:4.12}) converges in
$L^{2}(\Omega,\mathbb{P}^{x})$ to $0$, when
$n\longrightarrow\infty$, we define
$$V_n(z):=\int_{\R}h^{2}(y)\left(1_{\{\alpha(z)>\alpha_0\}}n^{1-\frac{\alpha(z)}{\alpha_0}}\frac{c(z)}{|y|^{\alpha(z)+1}}-\int_{\R/\Lambda}1_{\{\alpha(x)>\alpha_0\}}(x)n^{1-\frac{\alpha(x)}{\alpha_0}}\frac{c(x)}{|y|^{\alpha(x)+1}}\pi(dx)\right)dy$$
and proceed as for (\ref{eq:4.4}). It remains to prove that
(\ref{eq:4.9}) converges in $L^{2}(\Omega,\mathbb{P}^{x})$ to
$\tilde{C}_t^{0}$, when $n\longrightarrow\infty$. Let us define
$$V(z):=\int_{\R}h^{2}(y)\left(1_{\{\alpha(z)=\alpha_0\}}\frac{c(z)}{|y|^{\alpha_0+1}}-\int_{\R/\Lambda}1_{\{\alpha(x)=\alpha_0\}}\frac{c(x)}{|y|^{\alpha_0+1}}\pi(dx)\right)dy.$$
By $\tau$-periodicity of the functions $\alpha(x)$ and $c(x)$, the
function $V(z)$ is $\tau$-periodic and
$$\int_{\R/\Lambda}V(z)\pi(dz)=0.$$ Hence, by repeating the same calculation as
for (\ref{eq:4.4}), we have the claim. Therefore, by \cite[Theorem
VIII.2.17]{jacod}, we have proved that the sequence of processes
$\textbf{Y}^{p}_{n}$ converges in distribution to symmetric
$\alpha_0$-stable L\'{e}vy process $\textbf{L}$ with the compensator
(L\'{e}vy measure) $N^{0}(ds,dy).$

Now, let us  prove that the periodic stable-like chain
$\chain{X^{p}}$ is recurrent if and only if $\alpha_0\geq1$. By
\cite [Lemmas 2 and 3]{franke}, the set of recurrent paths $R(O)$ is
a continuity set for the probability measure
$\mathbb{P}^{x}_{\textbf{L}}(\cdot)$ for all $x\in\R$ and all open
bounded sets $O\subseteq\R$. Furthermore, since $\textbf{L}$ is a
$\lambda$-irreducible T-model (note that (\ref{eq:2.1}) is trivially
satisfied), by Proposition \ref{p2.4}, $\textbf{L}$ is recurrent if
and only if $\mathbb{P}^{x}_{\textbf{L}}(R(O))=1$ for all  $x\in\R$
and all open bounded sets $O\subseteq\R$, and it is transient if and
only if $\mathbb{P}^{x}_{\textbf{L}}(T(O))=1$ for all $x\in\R$ and
all open bounded sets $O\subseteq\R$.

Let $x\in\R$ be an arbitrary starting point and let $O\subseteq\R$
be an arbitrary open bounded set. By \cite[Theorem 2.1]{bil}, we
have
\begin{align}\label{eq:4.15}\lim_{n\longrightarrow\infty}\mathbb{P}^{x}_{\textbf{Y}^{p}_{n}}(R(O))=\mathbb{P}^{x}_{\textbf{L}}(R(O)).\end{align}
If the stable-like chain $\chain{X^{p}}$ is recurrent, since it is
 $\lambda$-irreducible T-model, it is H-recurrent as well. Hence,
by Proposition \ref{p2.7}, all the processes $\textbf{Y}^{p}_{n}$,
$n\in\N$, are H-recurrent. This implies
$$\mathbb{P}^{x}_{\textbf{Y}^{n}_{n}}(R(O))=1\quad\textrm{for
all}\quad n\in\N.$$ Therefore, by (\ref{eq:4.15}),
$\mathbb{P}_{\textbf{L}}(R(O)_x)=1$, i.e., $\textbf{L}$ is
recurrent.

Let us assume that  the periodic stable-like chain $\chain{X^{p}}$
is transient. Then, by Proposition \ref{p2.4},
$\mathbb{P}^{x}\left(\tau_O<\infty\right)=0$ for all $x\in\R$ and
all open bounded sets $O\subseteq\R$. Hence, by Proposition
\ref{p2.7}, $\mathbb{P}^{x}\left(\tau^{n}_O<\infty\right)=0$, i.e.,
$$\mathbb{P}^{x}_{\textbf{Y}^{p}_{n}}(R(O))=0$$ for all $n\in\N$, all $x\in\R$ and all open bounded sets $O\subseteq\R$. Therefore, by (\ref{eq:4.15}), $\mathbb{P}^{x}_{\textbf{L}}(R(O))=0$, i.e.,
$\textbf{L}$ is transient. Finally, by \cite[Corollary 37.17]{sato},
 $\textbf{L}$ is recurrent if and only if
$\alpha_0\geq1$. This accomplishes the proof.
\end{proof}

\begin{remark}\rm{\begin{enumerate}
                    \item [(i)] In Theorem \ref{tm1.2} we assume that the densities
$\{f_x\}_{x\in\R}$ satisfy $f_x(-y)=f_x(y)$ for all $x,y\in\R$ and
$f_x(y)\sim c(x)|y|^{-\alpha(x)-1},$ when
$|y|\longrightarrow\infty.$ This assumptions can be relaxed. Let
$\alpha:\R\longrightarrow(0,2)$ and $c_+,
c_-:\R\longrightarrow(0,\infty)$
  be Borel measurable
    functions   and let $\{f_x\}_{x\in\R}$ be an arbitrary family of  probability densities on $\R$.  Furthermore, let us assume that the function
$x\longmapsto f_x$ is a periodic function with period $\tau>0$ and
that the following conditions are satisfied:
\begin{description}
\item [\textbf{(PC1')}] the function $(x,y)\longmapsto f_x(y)$ is continuous and strictly positive;
\item [\textbf{(PC2')}] $f_x(y)\sim c_+(x)y^{-\alpha(x)-1}$,
when $y\longrightarrow\infty$, and

\ \  \ \ \ $f_x(y)\sim c_-(x)(-y)^{-\alpha(x)-1},$ when
$y\longrightarrow-\infty$, for all $x\in\R$;
                          \item [\textbf{(PC3')}]
 $\displaystyle\lim_{y\longrightarrow\infty}\sup_{x\in[0,\tau]}\left|f_x(y)\frac{y^{\alpha(x)+1}}{c_+(x)}-1\right|=0\quad \textrm{and}$
 $\displaystyle\lim_{y\longrightarrow-\infty}\sup_{x\in[0,\tau]}\left|f_x(y)\frac{|y|^{\alpha(x)+1}}{c_-(x)}-1\right|=0;$
\item [\textbf{(PC4')}] $\displaystyle\inf_{x\in
                          [0,\tau]}(c_-(x)\wedge c_+(x))>0$.
\end{description}
Hence, the densities $\{f_x\}_{x\in\R}$ have two-tail behavior. Let
$\chain{\bar{X}^{p}}$ be a Markov chain given by the transition
kernel $ \bar{p}(x,dy):=f_x(y-x)dy.$  By completely the same
arguments as in the proof of Theorem \ref{tm1.2}, we can deduce
recurrence and transience property of the chain
$\chain{\bar{X}^{p}}$. If the set
$\{x\in\R:\alpha(x)=\alpha_0:=\inf_{x\in\R}\alpha(x)\}$ has positive
Lebesgue measure, then by  subordination of the chain
$\chain{\bar{X}^{p}}$ with the Poisson process $\process{N}$ with
parameter $1$ (independent of the chain $\chain{\bar{X}^{p}}$), one
can prove that the process
$\{n^{-\frac{1}{\alpha_0}}\bar{X}^{p}_{N_{nt}}\}_{t\geq0}$ converges
in distribution, with respect to the Skorohod topology, to
$\alpha_0$-stable L\'{e}vy process. In general, this
$\alpha_0$-stable L\'{e}vy process is not  symmetric anymore.
Non-symmetry of the densities $\{f_x\}_{x\in\R}$ implies that the
$\alpha_0$-stable L\'{e}vy process has a nonzero shift parameter,
and two-tail behavior implies that the $\alpha_0$-stable L\'{e}vy
process has a nonzero skewness parameter. Hence, by \cite[Corollary
37.17]{sato}, the only recurrent cases are if either $\alpha_0>1$
and shift parameter  vanishes or $\alpha_0=1$ and skewness parameter
vanishes.
                    \item [(ii)] As already mentioned, it is shown  in \cite{franke}  that
if the functions $\alpha:\R\longrightarrow(0,2)$ and
$\gamma:\R\longrightarrow(0,\infty)$ are continuously differentiable
with bounded derivative and periodic  and  if the set
$\{x\in\R:\alpha(x)=\alpha_0:=\inf_{x\in\R}\alpha(x)\}$ has positive
Lebesgue measure, then the  stable-like process with the symbol
$p(x,\xi)=\gamma(x)|\xi|^{\alpha(x)}$  is recurrent if and only if
$\alpha_0\geq1.$ In general, we cannot apply Theorem \ref{tm1.2} for
the discrete-time version of this stable-like process, i.e., for the
stable-like chain $\chain{X^{\alpha(x)}}$ (the chain given by
(\ref{eq:1.2})), since we do not have a proof that its transition
densities satisfy condition (PC3). But, by repeating the proof of
Theorem \ref{tm1.1} we deduce: If $\alpha:\R\longrightarrow(0,2)$
and $\gamma:\R\longrightarrow(0,\infty)$ are  continuously
differentiable  and periodic functions with bounded derivative  and
if the set $\{x\in\R:\alpha(x)=\alpha_0:=\inf_{x\in\R}\alpha(x)\}$
has positive Lebesgue measure, then the stable-like  chain
$\chain{X^{\alpha(x)}}$  is recurrent if and only if
$\alpha_0\geq1$.

Similarly, by repeating the proof of Theorem \ref{tm1.1}, we can
prove transience property of the discrete-time version of the
stable-like process considered in \cite{rene-wang-feller}, i.e., the
process given by the symbol $p(x,\xi)=\gamma(x)|\xi|^{\alpha(x)}$,
where $\alpha:\R\longrightarrow(0,2)$ and
$\gamma:\R\longrightarrow(0,\infty)$ are continuously differentiable
functions with bounded derivative and such that
$\limsup_{|x|\longrightarrow\infty}\alpha(x)<1$ and
$0<\inf_{x\in\R}\gamma(x)\leq\sup_{x\in\R}\gamma(x)<\infty$.
                  \end{enumerate}
}
\end{remark}

\section{Discrete state case}
In this section we derive the same recurrence and transience
criteria as in Theorems \ref{tm1.1} and \ref{tm1.2} for discrete
version of the stable-like chains $\chain{X^{(\alpha,\beta)}}$ and
$\chain{X^{p}}$ (the chains given by (\ref{eq:1.3}) and
(\ref{eq:1.4})).   Without loss of generality, we treat the case on
the state space $\ZZ$. Let $\alpha:\ZZ\longrightarrow(0,2)$ and
$c:\ZZ\longrightarrow(0,\infty)$ be  arbitrary functions and let
$\{f_{i}\}_{i\in\ZZ}$ be a family of probability functions on $\ZZ$
which satisfies $f_{i}(j)\sim c(i)|j|^{-\alpha(i)-1},$ when
$|j|\longrightarrow\infty$. Let $\chain{X^{d}}$ be a Markov chain on
$\ZZ$ given by the following transition kernel
$$p(i,j):=f_{i}(j-i).$$ The chain $\chain{X^{d}}$ can be
understood as a discrete version of the stable-like chain
$\chain{X}$, i.e., the probability functions $f_{i}(j)$ are discrete
versions of densities $f_{x}(y).$ It is clear that if $f_{i}(j)>0$
for all $i,j\in\ZZ$, then the chain $\chain{X^{d}}$ is irreducible.
Therefore, it is
                                  either
                                  recurrent or transient.
                                  If  the  following
conditions are satisfied
\begin{description}
\item[\textbf{(CD1)}] $f_i(j)\sim c(i)|j|^{-\alpha(i)-1},$ when
$|j|\longrightarrow\infty$, for all $i\in\ZZ$;
                                    \item [\textbf{(CD2)}]there exists $k\in \N$ such that
\begin{align*}\lim_{|j|\longrightarrow\infty}\sup_{i\in\{-k,\ldots,k\}^{c}}\left|f_{i}(j)\frac{|j|^{\alpha(i)+1}}{c(i)}-1\right|=0,\end{align*}

                                  \end{description}
 then the chain $\chain{X^{d}}$ is recurrent
                                    if $\liminf_{|i|\longrightarrow\infty}\alpha(i)>1$, and
                                    it is transient if
$\limsup_{|i|\longrightarrow\infty}\alpha(i)<1$ (see \cite{ja}).
Note that  conditions (CD1) and (CD2) also implies irreducibility of
the chain $\chain{X^{d}}$ in the case when $f_{i}(j)>0$ is not
satisfied for all $i,j\in\ZZ$.

\subsection{Step case}

Let $\chain{X^{d(\alpha,\beta)}}$ be a discrete version of the
stable-like chain $\chain{X^{(\alpha,\beta)}}$ given by
(\ref{eq:1.3}), i.e., a special case of the chain $\chain{X^{d}}$
given by the following step functions
 $$\alpha(i)=\left\{\begin{array}{cc}
                                                      \alpha, & i<0 \\
                                                      \beta,&
                                                      i\geq0
                                                    \end{array}\right. \quad \textrm{and} \quad c(i)=\left\{\begin{array}{cc}
                                                      c, & i<0 \\
                                                      d,&
                                                      i\geq0,
                                                    \end{array}\right.$$
                                                    where
                                                    $\alpha,\beta\in(0,2)$
                                                    and
                                                    $c,d\in(0,\infty)$.

Recall that a random walk $\chain{S}$   is \emph{attracted} to a
random variable $X$  if there exist sequences of real numbers
$\{A_n\}_{n\in\N}$ and $\{B_n\}_{n\in\N}$, $B_n>0$ for all $n\in\N$,
such that
$$\frac{S_n}{B_n}-A_n\stackrel{\hbox{\scriptsize d}}{\longrightarrow} X.$$ Here $\stackrel{\hbox{\scriptsize d}}{\longrightarrow}$ denotes convergence in distribution.
 Furthermore, if $A_n=0$ for all $n\in\N$, then we say that the random walk  $\chain{S}$ is \emph{strongly attracted} to $X$.
The random variable $X$ can only have a stable distribution
 (see \cite[Theorem 2.1.1]{ibragimov-linnik}). Now,
from \cite[Theorem 35.2]{gnedenko} which gives necessary and
sufficient conditions in order that a random walk $\chain{S}$ is
attracted to a random variable with stable distribution with the
index of stability $\alpha\in(0,2)$, we easily derive:
\begin{proposition}\label{p5.2}
 Let $\alpha\in(0,2)$ and $c\in(0,\infty)$ be arbitrary and let
$f_{(\alpha,c)}:\ZZ\longrightarrow\R$ be an arbitrary probability
function such that $f_{(\alpha,c)}(j)\sim c|j|^{-\alpha-1},$ when
$|j|\longrightarrow\infty$. Let us assume that
$f_{(\alpha,c)}(-j)=f_{(\alpha,c)}(j)$ holds for all $j\in\ZZ$ if
$\alpha=1$, and  $\sum_{j\in\ZZ}jf_{(\alpha,c)}(j)=0$ holds if
$\alpha>1$. Then the random walk $\chain{S}$ with the jump
distribution
$$\left(\begin{array}{ll}
                                                                 \ldots\ \ -1 \ \ \ \ \ \  \ \ \ \  \ 0 \ \ \ \ \ \ \ \ \  \ 1 \ \ \ \ \ldots \\
\ldots f_{(\alpha,c)}(-1) \ f_{(\alpha,c)}(0) \ f_{(\alpha,c)}(1) \
\ldots
                                                               \end{array}
\right)$$ is strongly  attracted to S$\alpha$S distribution.
\end{proposition}
From Proposition \ref{p5.2}, as a special case of    \cite[Theorem
                                                    2]{rogozin},
we
                                                    have:
                                                    \begin{theorem}If
                                                    the probability
                                                    functions
                                                    $f_{(\alpha,c)}(j):=f_i(j)$,
                                                    for $i<0$,
                                                    and
                                                    $f_{(\beta,d)}(j):=f_i(j)$,
                                                    for $i\geq0$,
                                                    appearing in the
                                                    definition of
                                                    the chain
                                                    $\chain{X^{d(\alpha,\beta)}}$,
                                                    satisfy
                                                    $f_{(\alpha,c)}(j)=f_{(\alpha,c)}(-j)$
                                                    and
                                                    $f_{(\beta,d)}(j)=f_{(\beta,d)}(-j)$
                                                    for all
                                                    $j\in\ZZ$,
                                                    then
                                                    the chain
                                                    $\chain{X^{d(\alpha,\beta)}}$
                                                    is recurrent
                                                    if
                                                    $\alpha+\beta>2$,
                                                    and it is
                                                    transient if
                                                    $\alpha+\beta<2.$\end{theorem}
Note that previous theorem does not say anything about the case when
$\alpha+\beta=2.$ This case is not covered by \cite{rogozin} and it
seems to be much more complicated.

\subsection{Periodic case}

In this subsection we consider a  discrete version of the periodic
stable-like chain $\chain{X^{p}}$  given by (\ref{eq:1.4}).
 Let  $\chain{X^{dp}}$ be  a Markov chain on $\ZZ$  given by
 $$\alpha(i)=\left\{\begin{array}{cc}
                                                      \alpha, & i\in 2\ZZ \\
                                                      \beta,&
                                                      i\in2\ZZ+1
                                                    \end{array}\right.\quad \textrm{and} \quad c(i)=\left\{\begin{array}{cc}
                                                      c, & i\in 2\ZZ \\
                                                      d,&
                                                      i\in2\ZZ+1,
                                                    \end{array}\right.$$where
                                                    $\alpha,\beta\in(0,2)$
                                                    and
                                                    $c,d\in(0,\infty)$,
                                                    and let us
                                                    assume that
                                                    probability
                                                    functions
                                                    $f_{(\alpha,c)}(j):=f_{2i}(j)$
                                                    and
                                                    $f_{(\beta,d)}(j):=f_{2i+1}(j)$,
                                                    $i\in\ZZ$,
                                                    satisfy $f_{(\alpha,c)}(-j)=f_{(\alpha,c)}(j)$ and $f_{(\beta,d)}(-j)=f_{(\beta,d)}(j)$ for all $j\in\ZZ$.
 Let us define the following stopping times
inductively:
                                                    $T_0^{\alpha}:=0$,
                                                    $T_0^{\beta}:=0$,
                                                    $T_n^{\alpha}:=\inf\{k>T_{n-1}^{\alpha}:X^{dp}_k\in2\ZZ\}$
                                                    and
                                                    $T_n^{\beta}:=\inf\{k>T_{n-1}^{\beta}:X^{dp}_k\in2\ZZ+1\},$
                                                    for $n\in\N$.
\begin{proposition}\label{p5.4}
                                                    $\mathbb{P}^{i}(T_n^{\alpha}<\infty)=\mathbb{P}^{i}(T_n^{\beta}<\infty)=1$
                                                   for all
                                                    $i\in\ZZ$ and
                                                    all
                                                     $n\in\N.$
\end{proposition}
 \begin{proof}
 Let us prove that $\mathbb{P}^{i}(T_n^{\alpha}<\infty)=1$
                                                    for all
                                                    $i\in\ZZ$ and all $n\in\N$
                                                    by induction.
Let $i\in\ZZ$ be arbitrary and let $n=1$.  We have
\begin{align*}\mathbb{P}^{i}(T_1^{\alpha}=\infty)=&\mathbb{P}^{i}(X^{pd}_k\in2\ZZ+1, \  \forall
k\in\N)=\displaystyle\lim_{k\longrightarrow\infty}\mathbb{P}^{i}(X^{dp}_l\in2\ZZ+1,\
1\leq l\leq k)\\
=&\displaystyle\lim_{k\longrightarrow\infty}\displaystyle\sum_{i_1\in2\ZZ+1}p(i,i_1)\displaystyle\sum_{i_2\in2\ZZ+1}p(i_1,i_2)
\ldots\displaystyle\sum_{i_{k-1}\in2\ZZ+1}p(i_{k-2},i_{k-1})p(i_{k-1},2\ZZ+1).\end{align*}
Note that $p(2i+1,2\ZZ+1)=\sum_{j\in2\ZZ}f_{(\beta,d)}(j)<1$ for all
$i\in\ZZ$. Therefore, if we put $C:=\sum_{j\in2\ZZ}f_{(\beta,d)}(j)$
and $C_i:=p(i,2\ZZ+1)$, we have
\begin{align*}\mathbb{P}^{i}(T_1^{\alpha}=\infty)=\displaystyle\lim_{k\longrightarrow\infty}C_iC^{k-1}=0,\end{align*}
i.e., $\mathbb{P}^{i}(T^{\alpha}_1<\infty)=1$. Let us assume that
$\mathbb{P}^{i}(T_{n-1}^{\alpha}<\infty)=1$ and let us prove that
$\mathbb{P}^{i}(T_n^{\alpha}<\infty)=1$. By denoting
$N:=T_{n-1}^{\alpha}$ and using strong Markov property we have
\begin{align*}\mathbb{P}^{i}(T_n^{\alpha}<\infty)=\mathbb{E}^{i}[\mathbb{E}^{i}[1_{\{T_1^{\alpha}<\infty\}}\circ\theta_N|\mathcal{F}_N]]
=\mathbb{E}^{i}[\mathbb{E}^{X_N}[1_{\{T^{\alpha}_1<\infty\}}]]=\sum_{j\in
2\ZZ}\mathbb{E}^{i}[1_{\{X_N=j\}}]=1,\end{align*} where $\theta_n$
is the shift operator on the canonical state space
$\ZZ^{\{0,1,\ldots\}}$. In the completely analogously way we prove
that $\mathbb{P}^{i}(T_n^{\beta}<\infty)=1$
                                                    for all
                                                    $i\in\ZZ$ and all $n\in\N$.\end{proof}
For  $n\geq0$, let us put $Y^{\alpha}_n=X^{dp}_{T_n^{\alpha}}$ and
$Y^{\beta}_n=X^{dp}_{T_n^{\beta}}$, then, from  Proposition
\ref{p5.4},
 $\chain{Y^{\alpha}}$ and $\chain{Y^{\beta}}$ are well defined  Markov chains.
Let $i\in\ZZ$ and let us define the following stopping times:
$\tau_i:=\inf\{n\geq1: X^{dp}_n=i\}$,
$\tau^{\alpha}_i:=\inf\{n\geq1: Y^{\alpha}_n=i\}$ and
$\tau^{\beta}_i=\inf\{n\geq1: Y^{\beta}_n=i\}$.
\begin{proposition}\label{p5.5} For all $i\in\ZZ$,  $n\in\N$,  $j_1,\ldots,j_n\in2\ZZ$ and all $k_1,\ldots,k_n\in2\ZZ+1$ we have $\mathbb{P}^{i}(Y_1^{\alpha}=j_1,\ldots, Y_n^{\alpha}=j_n)>0$ and $\mathbb{P}^{i}(Y_1^{\beta}=k_1,\ldots, Y_n^{\beta}=k_n)>0$. In particular, the chains $\chain{Y^{\alpha}}$ and $\chain{Y^{\beta}}$ are irreducible on their state spaces.
\end{proposition}
\begin{proof} The set $2\ZZ$ is the state space of the chain $\chain{Y^{\alpha}}$,
and the set $2\ZZ+1$ is the state space of the chain
$\chain{Y^{\beta}}$ . Let $i\in\ZZ$ and $j_1\in2\ZZ$ be arbitrary,
then we have
\begin{align*}\mathbb{P}^{i}(Y_1^{\alpha}=j_1)=&p(i,j_1)+\displaystyle\sum_{i_1\in2\ZZ+1}p(i,i_1)p(i_1,j_1)+\displaystyle\sum_{i_1\in2\ZZ+1}p(i,i_1)\displaystyle\sum_{1_2\in2\ZZ+1}p(i_1,i_2)p(i_2,j_1)+\ldots\\
\geq&\displaystyle\sum_{i_1\in2\ZZ+1}p(i,i_1)p(i_1,j_1).\end{align*}
If $i\in2\ZZ$, then we take $i_1\in2\ZZ+1$ such that
$f_{(\alpha,c)}(i_1-i)>0$ and $f_{(\beta,d)}(j_1-i_1)>0$. Therefore,
$$\mathbb{P}^{i}(Y_1^{\alpha}=j_1)\geq f_{(\alpha,c)}(i_1-i)f_{(\beta,d)}(j_1-i_1)>0.$$ If
$i\in2\ZZ+1$, then we take $i_1\in2\ZZ+1$ such that
$f_{(\beta,d)}(i_1-i)>0$ and $f_{(\beta,d)}(j_1-i_1)>0$.  Hence, we
have
$$\mathbb{P}^{i}(Y_1^{\alpha}=j_1)\geq f_{(\beta,d)}(i_1-i)f_{(\beta,d)}(j_1-i_1)>0.$$
Let $i\in\ZZ$ and $j_1,j_2\in2\ZZ$ be arbitrary, then we have
$$\mathbb{P}^{i}(Y_1^{\alpha}=j_1,Y_2^{\alpha}=j_2)=\mathbb{P}^{i}(Y_2^{\alpha}=j_2|Y_1^{\alpha}=j_1)\mathbb{P}^{i}(Y_1^{\alpha}=j_1)=\mathbb{P}^{j_1}(Y_1^{\alpha}=j_2)\mathbb{P}^{i}(Y_1^{\alpha}=j_1)>0.$$
Let $n>2$. Let us suppose that for all $i\in\ZZ$ and for all
$j_1,\ldots j_{n-1}\in2\ZZ$ we have
$$\mathbb{P}^{i}(Y_1^{\alpha}=j_1,\ldots,Y^{\alpha}_{n-1}=j_{n-1})>0.$$ Let  $j_n\in2\ZZ$
be arbitrary, then we have
\begin{align*}\mathbb{P}^{i}(Y_1^{\alpha}=j_1,\ldots,Y^{\alpha}_{n}=j_{n})=&\mathbb{P}^{i}(Y_n^{\alpha}=j_n|Y_{n-1}^{\alpha}=j_{n-1},\ldots,
Y_1^{\alpha}=j_1)\mathbb{P}^{i}(Y_1^{\alpha}=j_1,\ldots,Y_{n-1}^{\alpha}=j_{n-1})\\=&\mathbb{P}^{j_{n-1}}(Y_1^{\alpha}=j_{n})\mathbb{P}^{i}(Y_1^{\alpha}=j_1,\ldots,Y^{\alpha}_{n-1}=j_{n-1})>0.\end{align*}
Analogously we prove the claim for the chain $\chain{Y^{\beta}}$.
Let $i,j\in 2\ZZ$ be arbitrary, then we have
 $$\mathbb{P}^{i}(\tau_j^{\alpha}<\infty)\geq
 \mathbb{P}^{i}(\tau_j^{\alpha}=1)=\mathbb{P}^{i}(Y_1^{\alpha}=j)>0.$$
Similarly, for arbitrary $i,j\in2\ZZ+1$ we have
 $$\mathbb{P}^{i}(\tau_j^{\beta}<\infty)>0.$$ Hence, the chains $\chain{Y^{\alpha}}$ and $\chain{Y^{\beta}}$
 are irreducible.
\end{proof}

\begin{proposition}\label{p5.6} The Markov chains $\chain{X^{dp}}$, $\chain{Y^{\alpha}}$ and $\chain{Y^{\beta}}$ have the same recurrence
property.
\end{proposition}
\begin{proof} Let $i\in2\ZZ$ be arbitrary, then we have
$$\mathbb{P}^{i}(\tau^{\alpha}_i=\infty)=\mathbb{P}^{i}(Y^{\alpha}_n\in2\ZZ\setminus\{i\},\ n\in\N)=\mathbb{P}^{i}(X^{dp}_n\in\ZZ\setminus\{i\},\
n\in\N)=\mathbb{P}^{i}(\tau_i=\infty).$$ Similarly, for arbitrary
$i\in2\ZZ+1$ we have
$\mathbb{P}^{i}(\tau_i=\infty)=\mathbb{P}^{i}(\tau^{\beta}_i=\infty).$
\end{proof}
\begin{proposition}\label{p5.7}Chains $\chain{Y^{\alpha}}$ and $\chain{Y^{\beta}}$ are symmetric  random walks with jump distributions $\mathbb{P}^{0}(Y_1^{\alpha}\in\cdot)$
and $\mathbb{P}^{1}(Y_1^{\beta}-1\in\cdot)$.
\end{proposition}
\begin{proof} Note first that for arbitrary $i,j\in\ZZ$ we have
\begin{align*}\mathbb{P}^{0}(Y^{\alpha}_{n+1}=2i-2j|Y^{\alpha}_{n}=0)=&p(0,2i-2j)+\displaystyle\sum_{k_1\in2\ZZ+1}p(0,k_1)p(k_1,2i-2j)\\&+\displaystyle\sum_{k_1\in2\ZZ+1}p(0,k_1)\displaystyle\sum_{k_2\in2\ZZ+1}p(k_1,k_2)p(k_2,2i-2j)+\ldots\\=&p(2i,2j)+\displaystyle\sum_{k_1\in2\ZZ+1}p(2j,k_1+2j)p(k_1+2j,2i)\\&+\displaystyle\sum_{k_1\in2\ZZ+1}p(2j,k_1+2j)\displaystyle\sum_{k_2\in2\ZZ+1}p(k_1+2j,k_2+2j)p(k_2+2j,2i)+\ldots\\=&\mathbb{P}^{0}(Y^{\alpha}_{n+1}=2i|Y^{\alpha}_{n}=2j).\end{align*}
Let us prove that the random variables
$Y^{\alpha}_{n+1}-Y^{\alpha}_n$, $n\geq0$, are symmetric i.i.d.
random variables with respect to the probability measure
$\mathbb{P}^{0}(\cdot)$. Let $n\geq0$. Then we have
\begin{align*}\mathbb{P}^{0}(Y^{\alpha}_{n+1}-Y^{\alpha}_n=2i)=&\displaystyle\sum_{j\in\ZZ}\mathbb{P}^{0}(Y^{\alpha}_{n+1}=2i+2j,
\
Y^{\alpha}_n=2j)\\=&\displaystyle\sum_{j\in\ZZ}\mathbb{P}^{0}(Y^{\alpha}_{n+1}=2i+2j|Y^{\alpha}_n=2j)\mathbb{P}^{0}(Y^{\alpha}_n=2j)=\mathbb{P}^{0}(Y_1^{\alpha}=2i).\end{align*}
 Let  $n\geq1$. Then we have
\begin{align*}&\mathbb{P}^{0}(Y^{\alpha}_{n+1}-Y^{\alpha}_n=2i,\
Y^{\alpha}_{n}-Y^{\alpha}_{n-1}=2j)\\&=\displaystyle\sum_{k\in\ZZ}\mathbb{P}^{0}(Y^{\alpha}_{n+1}=2i+2j,\
Y^{\alpha}_n=2k,\
Y^{\alpha}_{n-1}=2k-2j)\\&=\displaystyle\sum_{k\in\ZZ}\mathbb{P}^{0}(Y^{\alpha}_{n+1}=2i+2k|Y^{\alpha}_n=2k)\mathbb{P}^{0}(Y^{\alpha}_{n}=2k|Y^{\alpha}_{n-1}=2k-2j)
\mathbb{P}^{0}(Y^{\alpha}_{n-1}=2k-2j)\\&=\mathbb{P}^{0}(Y^{\alpha}_1=2i)\mathbb{P}^{0}(Y^{\alpha}_1=2j)=\mathbb{P}^{0}(Y^{\alpha}_{n+1}-Y^{\alpha}_n=2i)\mathbb{P}^{0}(Y^{\alpha}_{n}-Y^{\alpha}_{n-1}=2j).\end{align*}
This proves that the random variables
$Y^{\alpha}_{n+1}-Y^{\alpha}_n$, $n\geq0$, are i.i.d. random
variables. Symmetry is obvious. Completely analogously we prove that
the random variables $Y^{\beta}_{n+1}-Y^{\beta}_n$, $n\geq0$, are
i.i.d. symmetric random variables with respect to the probability
measure $\mathbb{P}^{1}(\cdot)$.\end{proof}
\begin{proposition}\label{p5.8} If $\alpha\wedge\beta<1$, then the chain
 $\chain{X^{dp}}$ is transient.
\end{proposition}

\begin{proof}  Without loss of generality, let us suppose that $\alpha\wedge\beta=\alpha<1$. By Proposition \ref{p5.6}, it is enough to prove that the chain $\chain{Y^{\alpha}}$ is transient.
From Proposition \ref{p5.7} we know that the chain
$\chain{Y^{\alpha}}$ is symmetric random walk on $2\ZZ$ with respect
to the probability measure $\mathbb{P}^{0}(\cdot).$  For every
$i\in\ZZ$ we have
$$\mathbb{P}^{0}(Y_1^{\alpha}=2i)=p(0,2i)+\displaystyle\sum_{j\in2\ZZ+1}p(0,j)p(j,2i)+\ldots\geq
f_{(\alpha,c)}(2i).$$ Let $\varphi(\xi)$ be the characteristic
function of the distribution $\mathbb{P}^{0}(Y_1^{\alpha}\in\cdot)$.
From the symmetry property of the distribution
$\mathbb{P}^{0}(Y_1^{\alpha}\in\cdot)$, we have
                                                   $$\mathrm{Re}\left(\frac{1}{1-\varphi(\xi)}\right)=\frac{1}{\displaystyle\sum_{j\in\ZZ}(1-\cos (2j\xi))\mathbb{P}^{0}(Y_1^{\alpha}=2j)}\leq\frac{1}{\displaystyle\sum_{j\in\ZZ}(1-\cos (2j\xi))f_{(\alpha,c)}(2j)}.$$
                                                    Note that $\sum_{j\in\ZZ}\cos (2j\xi)f_{(\alpha,c)}(2j)$
                 is                                   the  Fourier
                 transform
                                                    of the symmetric
                                                    sub-probability
                                                    measure on
                                                    $2\ZZ$. Using
                                                    completely the
                                                    same arguments as
                                                    in \cite[page 88]{spitzer}, from
\cite[Theorem 3.2.9]{durrett} we
                                                    get the desired result.
 \end{proof}
Let $m\geq1$, $\alpha_0,\ldots,\alpha_{m-1}\in(0,2)$ and
$c_0,\ldots,c_{m-1}\in(0,\infty)$ be arbitrary. Let $\chain{X^{dp}}$
be a Markov chain on $\ZZ$ given by
$$\alpha(i)=\alpha_j\quad \textrm{and} \quad c(i)=c_j$$ for $i\equiv j\,  \mathrm{mod}\,(m)$, i.e., the
functions $\alpha:\ZZ\longrightarrow(0,2)$ and
$c:\ZZ\longrightarrow(0,\infty)$ are periodic functions with period
$m$. Furthermore, let us suppose that probability functions
$f_{(\alpha_i,c_i)}(j)$, $i=0,\dots,m-1$, satisfy
$f_{(\alpha_i,c_i)}(-j)=f_{(\alpha_i,c_i)}(j)$ for all $j\in\ZZ$ and
 $i=0,\dots,m-1$. Then, it is not hard to prove that Propositions \ref{p5.4},
\ref{p5.5}, \ref{p5.6} and \ref{p5.7}, except perhaps the symmetry
property of related chains (random walks) $\chain{Y^{\alpha_i}}$,
$i=0,\dots,m,$ are also valid in this periodic case. Therefore,
analogously as in Proposition \ref{p5.8} using
$$\mathrm{Re}
                                                    \left(\frac{1}{1-z}\right)=\frac{1-a}{(1-a)^{2}+b^{2}}\leq\frac{1}{1-a}$$  for
all
                                                    $z=a+ib\in\CC$
                                                    such that
                                                     $|z|\leq1$, we have:
\begin{theorem}\label{t5.9} If
$\alpha_0\wedge\alpha_1\wedge\dots\wedge\alpha_{m-1}<1$, then the
chain $\chain{X^{dp}}$ is transient.
\end{theorem}
Clearly, the above statement should be an if and only if statement,
i.e., there is no reason not to believe that
$\alpha_0\wedge\alpha_1\wedge\dots\wedge\alpha_{m-1}=1$ implies
recurrence of the chain $\chain{X^{dp}}$. But this case is not
covered by \cite{ja} and, again, it seems to be much more
complicated.

\section*{Acknowledgement} The author would like to thank Prof. Zoran
Vondra\v{c}ek   for many discussions on the topic and for helpful
comments on the presentation of the results.

\bibliographystyle{alpha}
\bibliography{References}

\end{document}